\documentclass[suppldata]{interact}
 \usepackage{color}
 \usepackage{bm}
\usepackage{epstopdf}
\usepackage{caption}
\usepackage{graphicx}
\usepackage{subcaption}
\usepackage{algorithm}
\usepackage{algorithmicx}%
\usepackage{algpseudocode}
\usepackage{listings}
\usepackage{lscape}  
\usepackage{multirow}%
\usepackage{makecell}  
\usepackage{mathrsfs}%
\usepackage{textcomp}%
\usepackage{rotating}
\usepackage{siunitx} 

\usepackage{amssymb}
\usepackage{amsfonts}
\usepackage{amsmath}
\usepackage{enumerate,amstext,graphicx,latexsym,amscd}
\usepackage[titletoc, title]{appendix}
\usepackage{bbding}
\usepackage[misc]{ifsym}
\usepackage{ifthen,calc,lastpage,listings}
\usepackage{fancyhdr}
\usepackage{footmisc}
\usepackage[colorlinks,linkcolor=blue,citecolor=blue]{hyperref}
\usepackage[capitalise, nosort]{cleveref}
\crefname{equation}{}{}
\crefname{lem}{Lemma}{Lemmas}
\crefname{thm}{Theorem}{Theorems}
\crefname{assum}{Assumption}{Assumptions}
  \usepackage{paralist}
  \usepackage{graphics} 
  \usepackage{epsfig} 
 \usepackage{epstopdf}
\usepackage{geometry}
\usepackage{dsfont}
\usepackage{float}
 \usepackage{booktabs,multirow}
 \usepackage{tcolorbox}
\usepackage[numbers,sort&compress]{natbib}
\bibpunct[, ]{[}{]}{,}{n}{,}{,}

\theoremstyle{plain}
\newtheorem{theorem}{Theorem}[section]
\newtheorem{lemma}[theorem]{Lemma}

\newtheorem{proposition}[theorem]{Proposition}

\theoremstyle{definition}
\newtheorem{definition}[theorem]{Definition}

\theoremstyle{remark}
\newtheorem{remark}{Remark}

\newtheorem{assum}{Assumption}

\begin{document}


\title{Barzilai-Borwein Diagonal Quasi-Newton Method for Nonconvex Multiobjective Optimization Problems}
\author{
\name{Hua Liu\textsuperscript{a}  Zhuoxin Fan\textsuperscript{a} \thanks{CONTACT Liping Tang. Email: \href{mailto:tanglipings@163.com}{tanglipings@163.com}} Liping Tang\textsuperscript{a} and Xinmin Yang\textsuperscript{a}}
\affil{\textsuperscript{a}National Center for Applied Mathematics in Chongqing, Chongqing Normal University, Chongqing, China}
}

\maketitle
\let\thefootnote\relax\footnotetext{\textcolor{red}{
	This article was submitted to \textit{Optimization} on May 29, 2026.}}
\begin{abstract}
Currently, quasi-Newton methods for nonconvex multiobjective optimization have been developed to use a single matrix to approximate the Hessian matrix of the aggregated objective functions. 
Although this strategy achieves a certain balance between approximation accuracy and the computational cost of subproblems, each iteration requires storing a dense $n \times n$ matrix and performing matrix-vector multiplications, leading to substantial storage and computational cost, particularly in large-scale settings, and significantly reduces the overall computational efficiency of the method.
To overcome this limitation, we propose a Barzilai--Borwein diagonal quasi-Newton (BB-DQN) method for nonconvex multiobjective optimization. The proposed method employs a shared diagonal matrix to approximate the Hessian matrix of the aggregated objective functions, where each diagonal entry is updated independently according to the Barzilai--Borwein criterion and restricted within adaptive bounds to ensure positive definiteness. Without imposing convexity assumptions, we establish the global convergence of the BB-DQN method and further prove its $R$-linear convergence rate under mild conditions. Numerical results demonstrate that, compared with existing quasi-Newton methods for nonconvex multiobjective optimization, the proposed BB-DQN method achieves superior performance in terms of computational time and iteration number, especially for large-scale problems.

\end{abstract}

\begin{keywords}
Nonconvex multiobjective optimization; quasi-Newton method; Barzilai-Borwein method; diagonal Hessian approximation; global convergence; $R$-linear convergence

\end{keywords}
\section{Introduction}
In this paper, we consider the following unconstrained multiobjective optimization problem:
\begin{equation}\tag{MOP}\label{MOP}
	\min_{x\in\mathbb{R}^{n}} f(x)=(f_{1}(x), f_{2}(x), \cdots, f_{m}(x))^{\top}
\end{equation}
where $f_{i}(x):\mathbb{R}^{n}\rightarrow\mathbb{R}$ is continuously differentiable for each $i\in\{1, 2, \cdots, m\}$.

Multiobjective optimization problems (MOPs) have attracted considerable attention due to their wide range of applications in many fields, including resource management \cite{zhang2022multi,zhang2025multi}, intelligent transportation \cite{wen2024ensemble,ma2025multi}, and machine learning \cite{qu2021comprehensive,liu2024stochastic}. 
Unlike single-objective optimization, MOPs involve the simultaneous optimization of multiple conflicting objective functions and are therefore well suited for modeling decision-making problems with competing criteria.

Over the past two decades, gradient-based methods for MOPs have been extensively studied because they possess strong convergence guarantees and do not require predefined preference parameters. 
In 2000, Fliege and Svaiter \cite{fliege2000steepest} proposed the steepest descent (SD) method for multiobjective optimization, where the steepest descent direction is obtained as the optimal solution of the following subproblem:
\begin{equation*}
	\min_{d \in \mathbb{R}^n}
	\max_{i=1, \ldots, m}
	 \nabla f_i(x)^\top d 
	+\frac{1}{2}\|d\|^2.
\end{equation*}
Since then, a variety of first-order methods have been developed by employing different strategies for computing the descent direction or selecting the stepsize. 
Tanabe et al. \cite{tanabe2019proximal} extended the proximal gradient framework to composite MOPs. 
Morovati et al. \cite{morovati2016barzilai} incorporated the Barzilai--Borwein (BB) rule into multiobjective optimization, avoiding expensive line searches by employing spectral-like stepsizes. 
Lucambio P\'erez and Prudente \cite{lucambio2018nonlinear} investigated nonlinear conjugate gradient methods combined with Wolfe-type line search conditions to improve numerical stability. 
More recently, Chen et al. \cite{chen2023barzilai} proposed a Barzilai--Borwein descent (BBD) method for multiobjective optimization, which alleviates the small-stepsize issue of SD method caused by imbalances among objective functions through adaptive gradient scaling in the direction-finding subproblem. 

Despite their relatively low per-iteration cost, first-order methods often exhibit slow convergence, especially for ill-conditioned or large-scale problems. 
To improve convergence behavior, second-order and quasi-Newton methods for MOPs have also been developed. 
Fliege et al. \cite{fliege2009newton} proposed a Newton-type method for unconstrained MOPs and established local superlinear convergence. 
However, the method requires computing the exact Hessian matrix of each objective function and the inverse of the aggregated Hessian matrix, resulting in expensively computational costs.

To reduce these costs while retaining favorable convergence properties, several quasi-Newton methods have been proposed by replacing exact Hessians with suitable approximations. 
Povalej \cite{povalej2014quasi} developed a BFGS-type method for strongly convex unconstrained MOPs by constructing independent Hessian approximations for each objective function:
\begin{equation*}
	B_j^{k+1}
	=
	B_j^k
	-
	\frac{B_j^k s_k (s_k)^T B_j^k}
	{(s_k)^T B_j^k s_k}
	+
	\frac{y_j^k (y_j^k)^T}
	{(y_j^k)^T s_k},
	\quad j=1,2,\cdots,m,
\end{equation*}
where
$s_k=x_{k+1}-x_k$ and 
$y_j^k=\nabla f_j(x_{k+1})-\nabla f_j(x_k),~j=1,2,\cdots,m.$
Moreover, the local superlinear convergence rate of the proposed method was established in \cite{povalej2014quasi}. 
Morovati et al. \cite{morovati2018quasi} employed self-scaling BFGS and Huang-BFGS updates to approximate the Hessian matrix of each objective function and proved $R$-linear convergence rate. 
Later, Lapucci and Mansueto \cite{lapucci2023limited} proposed the first limited-memory quasi-Newton method for unconstrained MOPs by extending the classical L-BFGS two-loop recursion to the multiobjective setting using a single Hessian approximation matrix. Under strong convexity assumptions, they established global convergence and $R$-linear convergence to Pareto optimality. 
More recently, Peng et al. \cite{peng2026diagonal} proposed a diagonal BFGS quasi-Newton method for unconstrained multiobjective optimization, which uses a diagonal positive definite matrix to approximate the Hessian of each objective function and incorporates both Armijo and Wolfe line search strategies, establishing global convergence under mild assumptions without requiring convexity.

For nonconvex unconstrained MOPs, Prudente and Souza \cite{prudente2024global} proposed a globally convergent BFGS-type method by introducing a correction strategy to preserve the positive definiteness of the Hessian approximations. Specifically, they defined
\begin{equation*}
\gamma_j^k = y_j^k + r_j^k s_k, \forall j = 1, 2, \dots, m,
\end{equation*}
and updated the approximation matrices via
\begin{equation*}
	B_j^{k+1}
	=
	B_j^k
	-
	\frac{B_j^k s_k (s_k)^T B_j^k}
	{(s_k)^T B_j^k s_k}
	+
	\frac{\gamma_j^k (\gamma_j^k)^T}
	{(\gamma_j^k)^T s_k},
	\quad j=1,2,\cdots,m,
\end{equation*}
where	$y_j^k=\nabla f_j(x_{k+1})-\nabla f_j(x_k)$, $s_k=x_{k+1}-x_k$ and 
$$r_j^k = \max\{-\eta_j^k, 0\} + \vartheta_k \left\|\sum_{i=1}^m \mu_i^k \nabla f_i(x_k)\right\|,$$
with $\eta_j^k = \frac{(y_j^k)^\top s_k}{\|s_k\|^2}$, $\vartheta_k$ being a fixed constant and  $\sum_{i=1}^m \mu_i^k \nabla f_i(x_k)$ being in the convex hull of the gradients of the objective functions at $x_k$.
Although global convergence is guaranteed in the nonconvex setting, the method still requires constructing and storing multiple dense Hessian approximation matrices and solving quadratic subproblems at each iteration, resulting in high computational cost.

To improve computational efficiency, Yang et al. \cite{yang2025global} proposed a modified BFGS-type (M-BFGS) method for nonconvex multiobjective optimization by using a single approximation matrix for the aggregated Hessian information:
\begin{align*}
	B_{k+1} &= B_k - \frac{B_k s_k s_k^\top B_k}{s_k^\top B_k s_k} + \frac{\gamma_k \gamma_k^\top}{\gamma_k^\top s_k},
\end{align*}
where 
\begin{align*}
\gamma_k &= y_k + m_k s_k,\\
m_k &= \max\left\{ -\frac{y_k^\top s_k}{\{s_k\}^2}, 0\right\} + \sum_{i=1}^m \lambda_i^k \left( f_i(x_k) - f_i(x_{k+1}) \right), \\
y_k&=\sum_{i=1}^m \lambda_i^k ( \nabla f_i(x_{k+1}) - \nabla f_i(x_k)),\\
s_k &= x_{k+1} - x_k,
\end{align*}
and $\lambda^k$ is the Lagrange multipliers from the subproblem.
Subsequently, Hu et al. \cite{hu2026diagonal} presented a diagonal quasi-Newton method (D-QN) by approximating the aggregated Hessian matrix with a common diagonal positive definite matrix:
\begin{align*}
    B_{k+1} &= B_k - \frac{B_k s_k (s_k)^T B_k}{(s_k)^T B_k s_k} + \frac{u_k (u_k)^T}{(u_k)^T s_k},\\
	D_{k+1} &= \operatorname{diag}\big( (B_{k+1})_{ii} \big), \quad i = 1, \ldots, n,
\end{align*}
where 
\begin{align*}
    u_k& = \sum_{j=1}^m \lambda_j(y_j^k + t_j^k s_k),\\
	t_j^k &= \|\nabla f_j(x_k)\| + \max\left\{-\frac{(y_j^k)^T s_k}{\|s_k\|^2}, 0\right\}, \\
   y_j^k &= \nabla f_j(x_{k+1}) - \nabla f_j(x_k),\\
   s_k &= x_{k+1} - x_k,
\end{align*}
and $\lambda^k$ is the Lagrange multipliers from the subproblem.

These quasi-Newton methods for nonconvex multiobjective optimization achieve a compromise between approximation accuracy and the computational cost of solving the associated subproblems. Nevertheless, they suffer from two major limitations:
\begin{itemize}
  \item[(i)] Each iteration requires storing a dense $n \times n$ approximation matrix, resulting in substantial memory requirements, especially for large-scale problems;
  
  \item[(ii)] The matrix update involves matrix-vector products with computational complexity $\mathcal{O}(n^2)$, which may severely limit the efficiency of existing quasi-Newton algorithms in large-scale settings.
\end{itemize}

It is worth noting that the Barzilai-Borwein strategy has been shown to effectively capture local curvature information while maintaining very low computational complexity, even in multiobjective optimization settings \cite{chen2023barzilai}. Moreover, in the single-objective setting, Park et al. \cite{park2020diagonal} proposed a variable metric proximal gradient method, which employs a diagonal Barzilai-Borwein matrix as the variable metric to balance computational efficiency and geometric adaptability. The method preserves a low computational complexity of $\mathcal{O}(n)$ while effectively capturing scaling differences among coordinate directions, thereby achieving an ideal balance between low computational cost and good Hessian approximation.
These observations naturally motivate the following question:
\medskip

\emph{
Can the diagonal Barzilai-Borwein strategy be incorporated into the construction of diagonal Hessian approximation matrices for nonconvex multiobjective optimization, so as to develop efficient quasi-Newton method that reduce computational cost while preserving essential curvature information?
}

\medskip

In this paper, inspired by \cite{chen2023barzilai,park2020diagonal,yang2025global}, we propose a Barzilai-Borwein diagonal quasi-Newton (BB-DQN) method for nonconvex \eqref{MOP}. 
Based on the diagonal Barzilai-Borwein strategy \cite{park2020diagonal}, a single diagonal matrix is constructed to approximate the Hessian matrix of the aggregated objective functions.  We establish the global convergence of the proposed BB-DQN method without requiring convexity assumptions. Under mild conditions, the BB-DQN method is shown to achieve $R$-linear convergence rate. Numerical experiments further demonstrate the superior performance of the proposed method. The main contributions of this paper are summarized as follows:
\begin{itemize}
	\item A common diagonal matrix based on the diagonal Barzilai-Borwein strategy is introduced to simultaneously capture the local curvature information of the aggregated Hessian matrices associated with multiple objective functions. This design significantly reduces the computational and storage costs per iteration.
	
\item For nonconvex multiobjective optimization problems, the proposed BB-DQN method is  globally convergent without requiring convexity assumptions, and achieves $R$-linear convergence under mild conditions.

\item Numerical experiments on benchmark problems demonstrate the efficiency of the proposed method. Compared with existing quasi-Newton methods, such as those in \cite{yang2025global} and \cite{hu2026diagonal}, the BB-DQN method exhibits competitive performance in terms of computational time, number of iterations, function evaluations, and solution quality, while requiring significantly lower computational cost per iteration.  
\end{itemize}

The remainder of this paper is organized as follows. In Sect. \ref{sec:Preliminary}, we introduce notations and definitions used throughout the paper. Sect. \ref{sec3} reviews several multiobjective descent methods, including SD method, QN method, and MQN method. In Sect. \ref{sec4}, we present the proposed BB-DQN method and establish its convergence analysis. Numerical results are reported in Sect. \ref{sec5}, and concluding remarks are given in the final section.
\section{Preliminaries}\label{sec:Preliminary}

This section introduces the notation, definitions, and preliminary results used throughout the paper. Let $\mathbb{R}^n$ denote the $n$-dimensional Euclidean space, and let $\mathbb{R}*+^n$ denote its nonnegative orthant. The interior of $\mathbb{R}*+^n$ is defined by
$$
\operatorname{int}(\mathbb{R}_+^n)
:=\left\{x=(x_1,\dots,x_n)^\top \in \mathbb{R}_+^n : x_i>0, \forall i\in[n]\right\}.
$$
For any vectors $x=(x_1,\dots,x_n)^\top$ and $y=(y_1,\dots,y_n)^\top \in \mathbb{R}^n$, we define the partial orders
\begin{align*}
x \geq y
&\quad \Longleftrightarrow \quad
x-y \in \mathbb{R}_+^n,\\
x > y
&\quad \Longleftrightarrow \quad
x-y \in \operatorname{int}(\mathbb{R}_+^n),
\end{align*}
Let $[m]:={1,2,\dots,m}$ denote the index set of integers from $1$ to $m$. The unit simplex in $\mathbb{R}^m$ is defined by
$$
\Delta_m
:=
\left\{
\lambda \in \mathbb{R}^m :
\sum_{i=1}^m \lambda_i = 1,;
\lambda_i \geq 0,\ \forall i\in[m]
\right\}.
$$
The notation $\|\cdot\|$ denotes the Euclidean norm in $\mathbb{R}^n$. For each function $f_i$, its gradient and Hessian matrix at a point $x\in\mathbb{R}^n$ are denoted by $\nabla f_i(x)\in\mathbb{R}^n$ and $\nabla^2 f_i(x)\in\mathbb{R}^{n\times n}$, respectively, while the Jacobian matrix of $F$ at $x$ is denoted by $JF(x)\in\mathbb{R}^{m\times n}$. Furthermore, for any $x\in\mathbb{R}^n$ and any symmetric positive definite matrix $B\in\mathbb{R}^{n\times n}$, we define
$$
\|x\|_B
:=\sqrt{x^\top Bx}.
$$
For symmetric matrices $A,B\in\mathbb{R}^{n\times n}$, the notation
$
A \preceq (\prec), B
$
means that
$$
B-A \succeq (\succ), 0.
$$

To proceed with the subsequent analysis, we first recall several fundamental concepts of Pareto optimality  to (\ref{MOP}):

\begin{definition}\cite{fliege2000steepest}
	A point \(x^* \in \mathbb{R}^n\) is called a (weakly) Pareto optimal solution of (\ref{MOP}) if there exists no \(x \in \mathbb{R}^n\) such that \(f(x) \leq f(x^*)\) and \(f(x) \neq f(x^*)\) (or \(f(x) < f(x^*)\)).
\end{definition}

\begin{definition}\cite{fliege2000steepest}
	A point \(x^* \in \mathbb{R}^n\) is a local (weakly) Pareto optimal solution if there exists a neighborhood \(U \subseteq \mathbb{R}^n\) of \(x^*\) such that \(x^*\) is (weakly) Pareto optimal on \(U\).
\end{definition}

\begin{definition}\cite{fliege2000steepest}
	A point \(x^* \in \mathbb{R}^n\) is a Pareto critical point of (\eqref{MOP}) if for every direction \(d \in \mathbb{R}^n\), there exists an index \(i \in [m]\) such that \(\langle \nabla f_i(x^*), d \rangle \geq 0\).
\end{definition}

\begin{definition}\cite{fliege2000steepest}
	A vector \(d \in \mathbb{R}^n\) is a descent direction for \(f\) at \(x\) if \(\langle \nabla f_i(x), d \rangle < 0\) for all \(i \in [m]\).
\end{definition}

We now present the following lemma, which establishes the relationships among the three notions of Pareto optimality.

\begin{lemma}[\cite{fliege2009newton}, Theorem 3.1]
Assume that the objective function $f$ is continuously differentiable. Then the following statements hold:
\begin{enumerate}[(i)]
\item If $x^{*}$ is a local weakly Pareto optimal solution of \eqref{MOP}, then $x^{*}$ is a Pareto critical point of \eqref{MOP}.

\item If $f$ is convex and $x^{*}$ is a Pareto critical point of \eqref{MOP}, then $x^{*}$ is a weakly Pareto optimal solution of \eqref{MOP}.

\item If $f$ is strongly convex and $x^{*}$ is a Pareto critical point of \eqref{MOP}, then $x^{*}$ is a Pareto optimal solution of \eqref{MOP}.
\end{enumerate}
\end{lemma}

\section{Descent Methods for Multiobjective Optimization Problems}\label{sec3}
The steepest descent method for \eqref{MOP}, introduced by Fliege and Svaiter \cite{fliege2000steepest}, serves as the fundamental descent approach. For a given point $x \in \mathbb{R}^{n}$, the search direction is obtained by solving the subproblem:
\begin{equation}\label{minmax}
	\min_{d \in \mathbb{R}^{n}} \max_{i \in [m]} ~\nabla f_i(x)^\top d + \frac{1}{2} \| d \|^{2},
\end{equation}
which can be transformed into the following equivalent optimization problem (\ref{2}).
\begin{equation}\label{2}
	\begin{aligned}
		\min_{(t, d) \in \mathbb{R} \times \mathbb{R}^n} \quad & t + \frac{1}{2} \| d \|^2 \\
		\text{s.t.} \qquad & \nabla f_i(x)^\top d \leqslant t, \quad i = 1, \ldots, m.
	\end{aligned}
\end{equation}
Since (\ref{2}) is convex with linear constraints, strong duality holds. Consider that the dual problem of problem (\ref{2}):
$$
\max_{\lambda \in \Delta_{m}} \min_{(t, d) \in \mathbb{R} \times \mathbb{R}^n} t + \frac{1}{2} \| d \|^{2} + \sum_{i \in [m]} \lambda_{i} \left( \nabla f_i(x)^\top d - t \right).
$$
According to the Karush-Kuhn-Tucker (KKT) condition, we get
 \begin{align}
& d + \sum_{i \in [m]} \lambda_i \nabla f_i(x) = 0, \label{eq:KKT_d} \\
& \sum_{i \in [m]} \lambda_i = 1, \\
& \lambda_i \big( \nabla f_i(x)^\top d - t \big) = 0, \quad i \in [m],\\
& \nabla f_i(x)^\top d\rangle \le t, \quad i \in [m], \\
& \lambda_i \ge 0, \quad i \in [m].
\end{align}
From \eqref{eq:KKT_d}, the steepest descent direction is given by
\begin{equation}\label{dSD}
	d_{\mathrm{SD}}(x) = - \sum_{i \in [m]} \lambda_{i}^{\mathrm{SD}}(x) \nabla f_{i}(x),
\end{equation}
where $\lambda_{\mathrm{SD}}(x) = \left( \lambda_{1}^{\mathrm{SD}}(x), \lambda_{2}^{\mathrm{SD}}(x), \cdots, \lambda_{m}^{\mathrm{SD}}(x) \right)^{\top}$ is the solution of the following dual problem
\begin{equation}\label{4}
	- \min_{\lambda \in \Delta_{m}} \frac{1}{2} \left\| \sum_{i \in [m]} \lambda_{i} \nabla f_{i}(x) \right\|^{2}.
\end{equation}

\begin{proposition}[\cite{fliege2000steepest}]\label{le2}
	Let $\theta_{\mathrm{SD}}(x)$ and $d_{\mathrm{SD}}(x)$ denote the optimal value and the optimal solution of problem (\ref{minmax}), respectively. The following statements hold:
	\begin{enumerate}
		\item[(i)] If \( x \) is a Pareto critical point of problem (\ref{minmax}), then \( d_{\mathrm{SD}}(x) = 0 \) and \(\theta_{\mathrm{SD}}(x) = 0 \);
		\item[(ii)] If \( x \) is not a Pareto critical point of problem (\ref{minmax}), then \( d_{\mathrm{SD}}(x) \neq 0 \), \( \theta_{\mathrm{SD}}(x)<0 \), and 
$$\mathcal{D}(x, d_{SD}(x)) := \max_{i \in [m]}\nabla f_{i}(x)^{\top}d_{SD}(x) \leqslant -\frac{\|d_{SD}(x)\|^{2}}{2};$$
		\item[(iii)] The mapping  $d_{\mathrm{SD}}(\cdot)$ is continuous.
	\end{enumerate}
\end{proposition}

A modified quasi-Newton (MQN) method for solving \eqref{MOP}, proposed by Ansary and Panda \cite{ansary2015modified}, employs a single positive definite matrix to approximate all Hessian matrices, thereby significantly reducing the computational complexity. The modified quasi-Newton direction $d(x)$ is obtained as the optimal solution of the following subproblem:
\begin{equation}\label{mqn}
	\min_{d \in \mathbb{R}^{n}} \max_{i \in [m]} \nabla f_{i}(x)^{\top}d + \frac{1}{2}d^{\top}B(x)d,
\end{equation}
where $B(x) \in \mathbb{R}^{n \times n}$  is an approximation of the Hessian matrix $\sum_{i \in [m]} \lambda_i\nabla^2 F_i(x)$
and $B(x) \succ 0$. Similar to the analysis in the previous subsection, subproblem (\ref{mqn}) is equivalent to the following convex quadratic optimization problem:
\begin{equation*}
	\begin{aligned}
		\min_{(t,d) \in \mathbb{R} \times \mathbb{R}^{n}} \quad &t \\
		\text{s.t. }\qquad & \nabla f_{i}(x)^{\top}d + \frac{1}{2}d^{\top}B(x)d \leqslant t, \quad \forall i \in [m].
	\end{aligned}
\end{equation*}
By using the KKT conditions, the modified quasi-Newton direction is obtained as
\begin{equation}\label{d}
	d(x) = -B(x)^{-1}\left(\sum_{i=1}^{m}\lambda_{i}(x)\nabla f_{i}(x)\right),
\end{equation}
where $\lambda(x) = (\lambda_{1}(x), \lambda_{2}(x), \cdots, \lambda_{m}(x))^{\top}$ is the solution of the following dual problem
\begin{equation}\label{-min}
	-\min_{\lambda \in \Delta_{m}} \frac{1}{2}\left\|\sum_{i \in [m]}\lambda_{i}\nabla f_{i}(x)\right\|_{B(x)^{-1}}^{2}.
\end{equation}
Then the optimal value of subproblem (\ref{mqn}) is given by
\begin{equation}\label{theta}
	\theta(x) = -\frac{1}{2}d(x)^{\top}B(x)d(x).
\end{equation}
Moreover, we can get
\begin{equation}\label{D}
	\mathcal{D}(x,d(x)) := \max_{i \in [m]}\nabla f_{i}(x)^{\top}d(x) =-\|d(x)\|_{B(x)}^{2} \leqslant -\frac{1}{2}\|d(x)\|_{B(x)}^{2}.
\end{equation}

\begin{proposition}[\cite{ansary2015modified}, Theorem 4.2]\label{le3}
	Let \( \theta(x) \) and \( d(x) \) be the optimal value and solution of problem (\ref{mqn}), respectively. Then, the following statements hold.
		\begin{enumerate}
		\item[(a)] \( d(x) \neq 0 \) is a descent direction.
		
		\item[(b)] The following conditions are equivalent:
		\begin{enumerate}
			\item[(i)] The point \( x \in \mathbb{R}^{n} \) is not Pareto critical for (\ref{MOP});
			\item[(ii)] \( d(x) \neq 0 \);
			\item[(iii)] \( \theta(x) < 0 \).
		\end{enumerate}
	\end{enumerate}
\end{proposition}

\section{BB-DQN: Barzilai-Borwein Diagonal Quasi-Newton Method for Nonconvex Multiobjective Optimization Problems}\label{sec4}
In this section, we propose a novel quasi-Newton method for nonconvex multiobjective optimization problems (MOPs), called the Barzilai--Borwein Diagonal Quasi-Newton (BB-DQN) method.
At each iteration $k$, a single diagonal matrix $B_k$ is constructed via the diagonal Barzilai--Borwein strategy \cite{park2020diagonal} to approximate the aggregated Hessian matrix
$\sum_{i \in [m]} \lambda_i \nabla^2 F_i(x_k)$
associated with the objective functions.

\begin{algorithm}[H]
	\caption{BB-DQN: Barzilai-Borwein Diagonal Quasi-Newton Method for Nonconvex Multiobjective Optimization Problems}\label{alg:mbfgsmo2}
	\begin{algorithmic}[1]
		\Require $c_0 \in (0, 1]$, $c_1, c_2, \mu\in (0, \infty)$, $0 < \sigma_1 < \sigma_2 < 1$, $\varepsilon \geq 0$, $x_0 \in \mathbb{R}^n$, $B_0 > 0$, $\alpha_{-1} = [\alpha_{1}^{-1}, \ldots, \alpha_{n}^{-1}]^\top > 0$	
		
		\State Initialize $k \gets 0$.
		\State  Compute $\lambda_k = (\lambda_1^k, \lambda_2^k, \cdots, \lambda_m^k)^\top$ from problem (\ref{-min}), and compute $d_k$ from (\ref{d}).
		\If{$\|d_k\| < \varepsilon$}
		\State \Return $x_k$.
		\EndIf
		\State Compute $\omega_k := \min\{c_0, c_1 \|d_k \|^{c_2}\}$
		
		\State  Compute a step size $t_k > 0$ satisfying:
		\begin{align}
			f_{i}(x_{k}+t_{k}d_{k}) &\leqslant f_{i}(x_{k})+\sigma_{1} t_{k}\mathcal{D}(x_{k},d_{k}),\quad\forall i\in[m], \label{w1}\\
			\mathcal{D}(x_{k}+t_{k}d_{k},d_{k}) &\geqslant \sigma_{2}\mathcal{D}(x_{k},d_{k}),\label{w2}
		\end{align}
		where $\mathcal{D}(x_k, d_k) := \max_{i \in [m]} \nabla f_i(x_k)^\top d_k$.
		
		\State  Update: $x_{k+1} = x_k + t_k d_k$, $s_k = t_k d_k$, $\mu_i^k := \nabla f_i(x_{k+1}) - \nabla f_i(x_k)$, $y_k := \sum_{i=1}^m \lambda_i^k \mu_i^k$.
		
		\If{$\langle s_k, y_k \rangle > 0$ and $[\dfrac{y_k^\top s_k}{\|s_k\|^2}, \dfrac{\|y_k\|^2}{y_k^\top s_k}] \cap [\omega_k, \omega_k^{-1}] \neq \varnothing$}
		\State Compute $\alpha_k^- := \max\left\{{\dfrac{y_k^\top s_k}{\|s_k\|^2}}, {\omega_k}\right\}$, $\alpha_k^+ := \min\left\{\dfrac{\|y_k\|^2}{y_k^\top s_k}, \omega_k^{-1}\right\}$.
				
	\Else
	\State Compute $[\alpha_k^-, \alpha_k^+]= [\omega_k, \omega_k^{-1}].$
	
	\EndIf
	\State Compute $\alpha_{j}^{k} =
	\begin{cases}
		{\alpha_k^-}, & \frac{s_{j}^{k} y_{j}^{k} + \mu \alpha_{j}^{k-1}}{(s_{j}^{k})^{2} + \mu} < {\alpha_k^-}; \\[6pt]
		{\alpha_k^+}, & \frac{s_{j}^{k} y_{j}^{k} + \mu \alpha_{j}^{k-1}}{(s_{j}^{k})^{2} + \mu} > {\alpha_k^+}; \\[6pt]
		\frac{s_{j}^{k} y_{j}^{k} + \mu \alpha_{j}^{k-1}}{(s_{j}^{k})^{2} + \mu}, & \text{otherwise},
	\end{cases}$\\
	where $ s_{j}^{k} $ and $ y_{j}^{k} $ are $ j^{\text{th}} $ elements of $ s^{k} $ and $ y^{k} $ respectively.
	\State Compute $B_{k} = \text{Diag}(\alpha_{k}), \alpha_{k} = [\alpha_{1}^{k}, \ldots, \alpha_{n}^{k}]^\top. $ 	
	\State $k \gets k + 1$
	\State Go to Step 1
\end{algorithmic}
\end{algorithm}

\begin{remark}
In \textbf{Step 3}, the vector $d_k$ serves as the optimality measure, with the algorithm terminating if $\|d_k\| < \varepsilon$. 	
In \textbf{Step 6}, the parameter $\omega_k = \min\{c_0, c_1 \|d_k\|^{c_2}\}$ provides an adaptive lower bound. In fact, this ensures that the eigenvalues of the Hessian approximation matrix are uniformly bounded. 	
	In \textbf{Step 7}, the step size $t_k$ is determined by a Wolfe-type line search ensuring sufficient decrease and curvature for all objectives. 	
	\textbf{Steps 9--16} are the core of the BB-DQN method. Instead of maintaining full matrices, a single diagonal matrix $B_k = \operatorname{Diag}(\alpha_k)$ is constructed. Each entry $\alpha_j^k$ is computed via a safeguarded and regularized Barzilai--Borwein formula.
\end{remark}

\begin{assum}\label{a1}
The level set $\mathcal{L}_{f(x_0)} = \{ x \in \mathbb{R}^n \mid f(x) \le f(x_0) \}$ is bounded, where $x_0 \in \mathbb{R}^n$ is the given starting point.
\begin{remark}
	Assumption \ref{a1} implies that the function \( f \) is bounded below on the set \(\mathcal{L}_{f(x_0)}\), which we will use later.
\end{remark}
\end{assum}
Here, we demonstrate that for every  \( k \in \mathbb{N} \), the Wolfe line search in Algorithm \ref{alg:mbfgsmo2} is guaranteed to terminate after a finite number of steps.

\begin{proposition}\label{p1}
Suppose that Assumption \ref{a1} holds and \( d_{k} \) is a descent direction. Then there exists an interval \( [t_{l}, t_{u}] \) with \( 0 < t_{l} < t_{u} \) such that for all \( t \in [t_{l}, t_{u}] \), the inequalities (\ref{w1}) and (\ref{w2}) are satisfied.
\end{proposition}

\begin{proof}
The proof follows the same arguments as those used in Proposition 2 of \cite{lapucci2023limited} and is therefore omitted.
\end{proof}

Next, we examine the algorithm’s convergence without convex assumption; to proceed, we first state the requisite assumptions.
\begin{assum}\label{a3}
The gradient \( \nabla f_i \) is Lipschitz continuous with constant \( L_i > 0 \) on an open set \( C \) containing \(\mathcal{L}_{f(x_0)}\), i.e.,
\[
\| \nabla f_i(x) - \nabla f_i(y) \| \leq L_i \| x - y \| \quad \text{for all } x, y \in C \text{ and } i \in \left[m\right].
\]
\end{assum}
\begin{assum}\label{a4}
For any \( x \in \mathcal{L}_{f(x_0)}\), there exists constant \( a \) and \( b \) such that
\[aI_n \preceq B(x) \preceq bI_n.\]
\end{assum}

\begin{remark}
Assumption \ref{a4} is indeed reasonable because the diagonal entries of the Hessian approximation matrix $B_k$ lie between $\omega$ and $\omega^{-1}$. In fact, $a$ can be equal to $\min\{c_0, c_1 \varepsilon^{c_2}\}$, and $b$ can be equal to $\max\{c_0^{-1}, 1/(c_1 \varepsilon^{c_2})\}$.
\end{remark}

The next proposition establishes the relationship between $d_k$ and $d_{SD}^k$, and quantifies the per-iteration decrease of the objective f when algorithm is executed; throughout what follows we set \( L = \max_{i \in \langle m \rangle} L_i \).

\begin{proposition}\label{le4}
Suppose that Assumptions \ref{a3} and  \ref{a4} hold. Let \(\{x_k\}\) be the sequence produced by Algorithm \ref{alg:mbfgsmo2}. Then, for every \(i \in [m ]\) and all \( k \geq 0 \), we have
\begin{equation}\label{omega}
	\omega \| d_{SD}^k \|^2  \leq f_i(x_k) - f_i(x_{k+1}),
\end{equation}
where $\omega$ = \(\frac{\sigma_{1} (1-\sigma_{2}) a}{2b L}\).
\end{proposition}
\begin{proof}
Since Assumption \ref{a3} holds, \(\nabla f_{i}(x)\) is Lipschitz continuous with constant \(L_{i}\), \(i \in [m]\), we have
\begin{equation}
	(\nabla f_{i}(x_{k+1}) - \nabla f_{i}(x_{k}))^{\top}(x_{k+1} - x_{k}) \leqslant L_{i} \| x_{k+1}  - x_{k} \|^{2}.
\end{equation}
Combining this inequality with condition \eqref{w2} yields
\begin{align}
	(\sigma_{2} - 1)\mathcal{D}(x_{k}, d_{k})
	&\leqslant \mathcal{D}(x_{k+1}, d_{k}) - \mathcal{D}(x_{k}, d_{k}) \nonumber\\
	&\leqslant \max_{i \in [m]} \left( \nabla f_{i}(x_{k+1}) - \nabla f_{i}(x_{k}) \right)^{\top} d_{k}\nonumber \\
	&= \frac{1}{t_{k}} \max_{i \in [m]} \left( \nabla f_{i}(x_{k+1}) - \nabla f_{i}(x_{k}) \right)^{\top} (x_{k} + t_{k} d_{k} - x_{k}) \nonumber\\
	&= \frac{1}{t_{k}} \max_{i \in [m]} \left( \nabla f_{i}(x_{k+1}) - \nabla f_{i}(x_{k}) \right)^{\top} (x_{k+1} - x_{k}) \nonumber\\
	&\leqslant \frac{1}{t_{k}} \max_{i \in [m]} L_{i} \| x_{k+1} - x_{k} \|^{2} \nonumber\\
	&= t_{k} L \| d_{k} \|^{2},\label{Dl}
\end{align}
where \(L = \max\limits_{i \in [m]} L_{i}\). Then by Assumption \ref{a4} and inequality (\ref{D}), we have
\begin{align*}
	(1 - \sigma_{2})a \| d_{k} \|^{2}
	&\leqslant (1 - \sigma_{2}) \| d_{k} \|_{B_{k}}^{2} \\
	&\leqslant (\sigma_{2} - 1)\mathcal{D}(x_{k}, d_{k}) \\
	&\leqslant t_{k} L \| d_{k} \|^{2},
\end{align*}
where the last inequality follows from inequality \eqref{Dl}.  Hence,
\begin{equation}
	t_{k} \geqslant \frac{(1 - \sigma_{2})a}{L}.
\end{equation}
By the equation (\ref{theta}) and Assumption \ref{a4}, we get
\begin{align*}
	-\frac{1}{2}\|d_{k}\|_{B_{k}}^{2}
	&=\min_{d\in \mathbb{R}^{n}}\max_{i\in[m]}\nabla f_{i}(x_{k})^{\top}d+\frac{1}{2}\|d\|_{B_{k}}^{2}\\
	&\leqslant\min_{d\in \mathbb{R}^{n}}\max_{i\in[m]}\nabla f_{i}(x_{k})^{\top}d+\frac{b}{2}\|d\|^{2}\\
	&\leqslant\frac{1}{b}\min_{d\in \mathbb{R}^{n}}\max_{i\in[m]}\nabla f_{i}(x_{k})^{\top}(bd)+\frac{1}{2}\|bd\|^{2}\\
	&=-\frac{1}{2b}\|d_{SD}^{k}\|^{2}.
\end{align*}
Together with the Wolfe line search, we obtain
\begin{equation}
\begin{aligned}
	f_i(x_k) - f_i(x_{k+1}) & \geq -\sigma_1 t_k D(x_k, d_k) \\
	& \geq \frac{1}{2}\sigma_1 t_k {\|d_k\|_{B_k}^{2}}\\
	& \geq \ \frac{\sigma_{1} t_k}{2b} \|d_{SD}^k\|^2\\
	& \geq \ \frac{\sigma_{1} (1-\sigma_{2}) a}{2b L} \|d_{SD}^k\|^2.
\end{aligned}
\end{equation}
\end{proof}

The preceding analysis establishes the monotonic descent property of Algorithm \ref{alg:mbfgsmo2}. We now proceed to prove its convergence.

\begin{theorem}\label{l1}
Assume that Assumptions \ref{a1}, \ref{a3} and \ref{a4} hold. Let \(\{x_k\}\) be the sequence generated by Algorithm \ref{alg:mbfgsmo2}. Then, every accumulation point of the sequence \(\{x_k\}\) is Pareto critical for (\ref{MOP}).
\end{theorem}

\begin{proof}
From Assumption \ref{a1} and Proposition \ref{le4}, it follows that \( f \) is bounded below and the sequence \(\{ f_i(x_k) \}\), \(i\in[m]\) is monotonically decreasing. Consequently, we get 
$$\lim_{k \to \infty} (f_i(x_k) - f_i(x_{k+1})) = 0.$$
Combing it with (\ref{omega}), we get $$\lim_{k \to \infty} \| d_{\text{SD}}^k \| = 0.$$
Furthermore, Assumption \ref{a1} implies that there exists an infinite subsequence \(x_{k_t}\) such that 
\begin{center}
\( x_{k_t} \xrightarrow{k_t} x^* \) and \(\| d_{SD}^{k_t} \| \xrightarrow{k_t} 0 \) as $t\rightarrow\infty$. 
\end{center}
By Proposition \ref{le2}, the mapping \( d_{SD}(x) \) is continuous, and thus
\[d_{SD}(x^*) = 0.\]
It then follows from Proposition \ref{le3} that \( x^* \) is Pareto critical for (\ref{MOP}).
\end{proof}

Under the additional assumption that each objective component is strongly convex and $L$-smooth, we further establish that the sequence of iterates generated by Algorithm converges to a Pareto optimal solution at an R-linear rate.

\begin{assum}\label{a5}
There exist constants \( U \) and \( V \) satisfying \( 0 < U \leq 1 \leq V \) such that
\[
U \| z \|^2 \leq z^{\top} \nabla^2 f_i(x) z \leq V \| z \|^2, \quad \forall  x \in \mathcal{L}_{f(x_0)}, z \in \mathbb{R}^n, i \in [m].
\]
\end{assum}

\begin{remark}\label{r2}
Assumption \ref{a5} implies that Assumptions \ref{a1} and \ref{a3} hold. Moreover, Assumption \ref{a4} is also satisfied for the sequence \(\{x_k\}\) generated by Algorithm \ref{alg:mbfgsmo2} when $k$ is sufficiently large. 

Let $x \in \mathcal{L}$. Recall that Assumption \ref{a5} states that the operator $\nabla f_i(x)$ is $\mu$-strongly monotone and $L$-Lipschitz continuous for all $i \in [m]$. Consequently, $\sum_{i=1}^{m}\lambda_{i} \nabla f_i(x)$ is $\mu$-strongly monotone and $L$-Lipschitz continuous for $\lambda_i \in \Delta_{m}$ and $i \in [m]$, i.e.,
\begin{equation} \label{eq:strong_mono}
	\langle \sum_{i=1}^{m}\lambda_{i} \nabla f_i(x) - \sum_{i=1}^{m}\lambda_{i} \nabla f_i(y), x - y \rangle \geq U \|x - y\|^2, \quad \forall x, y \in \mathcal{L},
\end{equation}
and
\begin{equation} \label{eq:lipschitz}
	\|\sum_{i=1}^{m}\lambda_{i} \nabla f_i(x) - \sum_{i=1}^{m}\lambda_{i} \nabla f_i(y)\| \leq V \|x - y\|, \quad \forall x, y \in \mathcal{L}.
\end{equation}
By setting $x = x_{k+1}$ and $y = x_k$, we obtain that
\[
y_k^T s_k \geq U \|s_k\|^2 \quad \text{and} \quad y_k^T s_k \geq \frac{U}{V^2} \|y_k\|^2,
\]
where $y_k=\sum_{i=1}^{m}\lambda_{i} \nabla f_i(x) - \sum_{i=1}^{m}\lambda_{i} \nabla f_i(y)$ and $s_k=x_{k+1}-x_k$. Then we have
\[
U \leq \dfrac{y_k^\top s_k}{\|s_k\|^2} \leq \dfrac{\|y_k\|^2}{y_k^\top s_k} \leq \dfrac{V^2}{U}.
\]
Therefore, Assumption \ref{a4} is satisfied for \(\{x_k\}\) generated by Algorithm \ref{alg:mbfgsmo2} when $k$ is large enough, such that the $\omega_k$ becomes negligible.
\end{remark}

\begin{theorem}\label{t2}
Suppose that Assumptions \ref{a4} and \ref{a5} hold. Let \(\{x_k\}\) be the sequence generated by Algorithm \ref{alg:mbfgsmo2} and \(x^*\) be the limit point of the sequence. Then, \(\{x_k\}\) converges R-linearly to \(x^*\).
\end{theorem}

\begin{proof}
Based on Assumption \ref{a5} and the second-order Taylor expansion of \(f_i\) around \(x_k\), we have
\[
\frac{U}{2} \| x - x_k \|^2 \leq f_i(x) - f_i(x_k) - \nabla f_i(x_k)^T (x - x_k) \quad \text{for all } i \in [m].\]
Substituting \(x = x^*\)  and using (\ref{dSD}), we get that
\begin{equation}\label{17}
	\frac{U}{2} \| x^* - x_k \|^2 \leq \sum_{i=1}^{m} \lambda_i^{SD}(x_k) f_i(x^*) - \sum_{i=1}^{m} \lambda_i^{SD}(x_k) f_i(x_k) + (x^* - x_k)^T {d_{SD}^{k}} .
\end{equation}
It follows from Proposition \ref{le4} that
\[
\sum_{i=1}^{m} \lambda_i^{SD}(x_k) f_i(x^*) - \sum_{i=1}^{m} \lambda_i^{SD}(x_k) f_i(x_k) \leq 0.
\]
Combing this inequality with  (\ref{17}) and applying the Cauchy-Schwarz inequality, we derive 
\[
\frac{U}{2} \| x^* - x_k \|^2 \leq  (x^* - x_k)^T d_{SD}^{k} \leq  \| x^* - x_k \| \|d_{SD}^{k}\|,
\]
which implies 
\begin{equation}\label{18}
	\frac{U}{2} \| x^* - x_k \| \leq \| d_{SD}^{k} \|.
\end{equation}

On the other hand, the second-order Taylor expansion of \(F_i\) aroud \(x^*\), together with Assumption \ref{a5} implies 
\[
\frac{U}{2} \| x - x^* \|^2 \leq f_i(x) - f_i(x^*) - \nabla f_i(x^*)^T (x - x^*) \leq \frac{V}{2} \| x - x^* \|^2 \quad \text{for all } i \in [m].
\]

Setting \(x = x_k\) and taking a weighted sum \(\lambda^* \in \Delta_m\), we get
\begin{equation*}
	\sum_{i=1}^{m} \lambda_{i}^* f_i(x_k) - \sum_{i=1}^{m} \lambda_{i}^* f_i(x^*)-\sum_{i=1}^{m} \lambda_{i}^* \nabla f_i(x^*)^T (x - x^*) \leq \frac{V}{2} \| x_k - x^* \|^2.
\end{equation*}
By (\ref{dSD}) and  Proposition \ref{le2}, the middle term vanishes, yielding 
\begin{equation}\label{19}
 \sum_{i=1}^{m} \lambda_{i}^* f_i(x_k) - \sum_{i=1}^{m} \lambda_{i}^*f_i(x^*) \leq \frac{V}{2} \| x_k - x^* \|^2.
\end{equation}
Combining (\ref{18}) and \eqref{19}, we have
\begin{equation}\label{20}
	\sum_{i=1}^{m} \lambda_{i}^* f_i(x_k) - \sum_{i=1}^{m} \lambda_{i}^* f_i(x^*) \leq \frac{2V}{U^2} \| d_{SD}^{k} \|^2.
\end{equation}

Applying (\ref{omega}), we obtain
\[
\sum_{i=1}^{m} \lambda_{i}^* f_i(x_{k+1}) \leq \sum_{i=1}^{m} \lambda_{i}^* f_i(x_k) - \omega \| d_{SD}^{k} \|^2.
\]
Subtracting \(\sum_{i=1}^{m} \lambda_{i}^* f_i(x^*)\) from both sides of the above inequality and using (\ref{20}), it follows that
\begin{equation}\label{21}
	\begin{aligned}
		&\sum_{i=1}^{m} \lambda_{i}^* f_i(x_{k+1}) - \sum_{i=1}^{m} \lambda_{i}^* f_i(x^*)\\ \leq &\sum_{i=1}^{m} \lambda_{i}^* f_i(x_k) - \sum_{i=1}^{m} \lambda_{i}^* f_i(x^*) - \omega \| d_{SD}^{k} \|^2\\ \leq &\left( 1 - \frac{\omega U^2}{2V} \right) \left( \sum_{i=1}^{m} \lambda_{i}^* f_i(x_k) - \sum_{i=1}^{m} \lambda_{i}^* f_i(x^*) \right).
	\end{aligned}
\end{equation}
Under the choice $\omega$ = \(\frac{\sigma_{1} (1-\sigma_{2}) a}{2b L}\) specified in Proposition \ref{le4} and Remark \ref{r2}, it follows directly that  \(\omega U^2 / 2V \in (0, 1)\). Recursively applying  (\ref{21}), we get
\begin{align*}
	&\sum_{i=1}^{m} \lambda_{i}^* f_i(x^{k+1}) - \sum_{i=1}^{m} \lambda_{i}^* f_i(x^*) \\ \leq &\left( 1 - \frac{\omega U^2}{2V} \right)^{k+1} \left( \sum_{i=1}^{m} \lambda_{i}^* f_i(x^0) - \sum_{i=1}^{m} \lambda_{i}^* f_i(x^*) \right),
\end{align*}
which, combined with the left-hand side of (\ref{19}), gives
\[
\| x_{k+1} - x^* \| \leq \sqrt{\frac{2}{U} \left( \sum_{i=1}^{m} \lambda_{i}^* f_i(x^0) - \sum_{i=1}^{m} \lambda_{i}^* f_i(x^*) \right)} \cdot \left( 1 - \frac{\omega U^2}{2L} \right)^{(k+1)/2} .
\]
Therefore, the sequence \(\{x_k\}\) converges R-linearly to \(x^*\).
\end{proof}
	\section{Numerical Experiments}\label{sec5}
\label{sec:numerical_experiments}
This section presents some numerical experiments to illustrate the performance of the proposed BB-DQN method. We are particularly interested in verifying the effectiveness of 
updating the Hessian approximations at each iteration. For numerical comparison, we have considered the following methods:

\begin{itemize}
	\item\textbf{M-BFGS \cite{yang2025global}:} The Modified BFGS-type method for multiobjective optimization, in which the Hessian approximations are updated
	by:
	\begin{align}
		m_k &= \max\{-\frac{y_k^\top s_k}{\{s_k\}
			^2}, 0\} + \sum_{i=1}^m \lambda_i^k \left( f_i(x_k) - f_i(x_{k+1}) \right), \\
		B_{k+1} &= B_k - \frac{B_k s_k s_k^\top B_k}{s_k^\top B_k s_k} + \frac{\gamma_k \gamma_k^\top}{\gamma_k^\top s_k},
	\end{align}
	where $s_k := x_{k+1} - x_k, \mu_i^k := \nabla f_i(x_{k+1}) - \nabla f_i(x_k),  y_k := \sum_{i=1}^m \lambda_i^k \mu_i^k, \gamma_k = y_k + m_k s_k$.
	\item\textbf{D-QN \cite{hu2026diagonal}:} The diagonal Quasi-Newton method with a modified BFGS update.In this approach, the common diagonal positive definite matrix is utilized by
	\begin{align}
		D_{k+1} &= B_k - \frac{B_k s_k (s_k)^T B_k}{(s_k)^T B_k s_k} + \frac{u_k (u_k)^T}{(u_k)^T s_k},\\
		B_{k+1} &= \operatorname{diag}\big( (D_{k+1})_{ii} \big), \quad i = 1, \ldots, n,
	\end{align}
	where $u_k = \sum_{j=1}^m \lambda_j \eta_j^k$, $s_k = x_{k+1} - x_k$, $y_j^k = \nabla f_j(x_{k+1}) - \nabla f_j(x_k)$, for $j = 1, 2, \ldots, m$,
	\[
	\eta_j^k = y_j^k + t_j^k s_k, \quad t_j^k = \|\nabla f_j(x_k)\| + \max(-r_j^k, 0), \quad r_j^k = \frac{(y_j^k)^T s_k}{\|s_k\|^2}.
	\]
\end{itemize}

All algorithms were implemented in MATLAB 2018b. We used the Frank-Wolfe gradient method to solve the dual problem (\ref{-min}). The algorithms terminated when $\|d_k\| < \varepsilon$ or when the iteration count reached the maximum. All algorithms were implemented with identical parameters, the details of which are summarized in Table \ref{tab1:parameters}. The main characteristics of all test problems are summarized in Table \ref{tab1:problem_description} and \ref{tab2:problem_description}. The columns \texttt{n} and \texttt{m} represent the number of variables and objectives, respectively. The ‘Convex’ column specifies the convexity of the problem, with ‘\texttt{Y}’ denoting a convex MOP and ‘\texttt{N}’ denoting a nonconvex MOP. For each test problem, 200 runs were executed from randomly generated initial points within the bounds specified in Tables \ref{tab1:problem_description} and \ref{tab2:problem_description}. Performance was evaluated using the average CPU time in milliseconds (\texttt{time}), the average number of iterations (\texttt{iter}), the average number of function evaluations (\texttt{feval}), and the number of failure points (\texttt{NF}) that did not converge within the limit. We denote ‘Failed’ as ‘\texttt{F}’, which indicates that the total CPU time reached \num{2e5}s. The detailed numerical results are presented in Tables \ref{tab1:comparison_landscape} and \ref{tab3:comparison_landscape}.

To facilitate a clear comparison of the performance among the tested methods, we adopt the performance profiles introduced by Dolan and Mor{\'e}  \cite{dolan2002benchmarking}. Below, we provide a concise description of this tool.

Let $n_s$ denote the number of solvers and $n_p$ the number of test problems. For each solver $s \in \mathcal{S}$ and problem $p \in \mathcal{P}$, let $o_{p,s}$ represent the performance measure of solver $s$ on problem $p$. The performance ratio is defined as
\[
z_{p,s} = \frac{o_{p,s}}{\min\{o_{p,s} : s \in \mathcal{S}\}},
\]
and the corresponding cumulative distribution function $\rho : [1, \infty) \to [0, 1]$ is given by
\[
\rho(\tau) = \frac{|\{p \in \mathcal{P} : z_{p,s} \leq \tau\}|}{n_p}.
\]

The performance profile is then depicted by the cumulative distribution function $\rho$. Notably, $\rho(1)$ quantifies the probability that the solver outperforms all other solvers. In a performance profile, efficiency is reflected at the leftmost part of the graph (around $\tau = 1$), while robustness is characterized by the right-hand side of the graph.

Beyond the efficiency metrics, the quality of the obtained Pareto front approximation is critical. Following the methodology in \cite{custodio2011direct}, we employed the Purity, $\Gamma$, and $\Delta$ metrics for assessment, visualized via performance profiles.

\begin{itemize}
	\item \textbf{Purity}: Defined as $\tilde{t}_{p,s} = |F_{p,s} \cap F_{p}| / |F_{p}|$, where $F_{p,s}$ represent the approximate Pareto frontier obtained by solvers $s \in S$ for problem $p \in P$ and $F_p$ represent an approximation to the true Pareto frontier of problem $p$, computed by removing the dominant points from the set $\bigcup_{s \in S} F_{p,s}$. The purity metric. This metric measures the fraction of solutions obtained by an algorithm that belong to the true Pareto set; a higher value indicates better convergence.
	\item \textbf{$\Gamma$ metric}: Measures the maximum gap between consecutive solutions across all objective functions. Formally defined as
		\[
		\Gamma_{p,s} := \max_{i \in \{1,\dots,m\}} \max_{k \in \{0,\dots,N\}} \chi_{i,k},
		\]
		where $\chi_{i,k} = |f_i(x_{k+1}) - f_i(x_k)|$ represents the absolute difference between consecutive solutions along objective $i$. A lower $\Gamma$ value indicates a more uniform distribution without significant gaps in the frontier.
		
		\item \textbf{$\Delta$ metric}: Quantifies the relative deviation from a perfectly uniform spread, accounting for both intermediate solutions and the coverage of extreme points. It is defined as
		\[
		\Delta_{p,s} := \max_{i \in \{1,\dots , m\}} \left( \frac{\chi_{0,i} + \chi_{N,i} + \sum_{k=1}^{N-1} |\chi_{i,k} - \bar{\chi}_i|}{\chi_{0,i} + \chi_{N,i} + (N-1)\bar{\chi}_i} \right),
		\]
		where $\bar{\chi}_i = \frac{1}{N} \sum_{k=1}^{N} \chi_{k-1,i}$ is the average gap for objective $i$. A $\Delta$ value closer to 1 signifies a more even and well-distributed approximation, while a value near 0 indicates poor spread.
\end{itemize}
\begin{table}[htbp!]
	\centering
	\caption{The parameters for algorithms used in numerical experiments.}
	\label{tab1:parameters}
	\setlength{\tabcolsep}{10pt}
	\begin{tabular}{l c c c c c c c c}
		\toprule
		\textbf{Parameter} & $\varepsilon$ & \textbf{Max Iter} & $\sigma_1$ & $\sigma_2$ & $c_0$ & $c_1$ & $c_2$ & $B_0$ \\
		\midrule
		{Number} &$10^{-4}$ & 2000 & 0.01 & 0.9 & $10^{-4}$ & 1 & 3 & $I_n$ \\
		\bottomrule
	\end{tabular}
\end{table}

\begin{table}[htbp!]
	\centering
	\caption{The description of problems used in numerical experiments.}
	\label{tab1:problem_description}
	\begin{tabular}{@{}l c c l l c l@{}}
		\toprule
		\textbf{Problem} & \textbf{n} & \textbf{m} & \textbf{\(x_L\)} & \textbf{\(x_U\)} & \textbf{Convex} & \textbf{References} \\
		\midrule
		SLCDT1 & 2 & 2 & $(-1.5, -1.5)$ & $(1.5, 1.5)$ & N & \cite{schutze2008convergence} \\
		PNR & 2 & 2 & $(-1, -1)$ & $(1, 1)$ & Y & \cite{preuss2006pareto} \\
		MOP2 & 2 & 2 & $(-4, -4)$ & $(4, 4)$ & N & \cite{huband2006review} \\
		MOP5 & 2 & 3 & $(-1, -1)$ & $(1, 1)$ & N & \cite{huband2006review} \\
		MOP7 & 2 & 3 & $(-400, -400)$ & $(400, 400)$ & Y & \cite{huband2006review} \\
		Far & 2 & 2 & $(-1, -1)$ & $(1, 1)$ & N & \cite{huband2006review} \\
		KW2 & 2 & 2 & $(-1, -1)$ & $(1, 1)$ & N & \cite{kim2005adaptive} \\
		FF1 & 2 & 2 & $(-0.5, -0.5)$ & $(0.5, 0.5)$ & N & \cite{huband2006review} \\
		Deb & 2 & 2 & $(0.1, 0.1)$ & $(1, 1)$ & N & \cite{morovati2016barzilai} \\
		DD & 5 & 2 & $(-0.5, \ldots, -0.5)$ & $(0.5, \ldots, 0.5)$ & N & \cite{morovati2016barzilai} \\
		BK1 & 2 & 2 & $(-5, -5)$ & $(10, 10)$ & Y & \cite{huband2006review} \\
		WIT & 2 & 2 & $(-2, -2)$ & $(2, 2)$ & Y & \cite{chen2023barzilai} \\
		NT2a & 20 & 2 & $(-0.5, \ldots, -0.5)$ & $(0.5, \ldots, 0.5)$ & N & \cite{anh2025steepest} \\
		NT2b & 70 & 2 & $(-0.5, \ldots, -0.5)$ & $(0.5, \ldots, 0.5)$ & N & \cite{anh2025steepest} \\
		NT2c & 120 & 2 & $(-0.5, \ldots, -0.5)$ & $(0.5, \ldots, 0.5)$ & N & \cite{anh2025steepest} \\
		ZKG7 & 2 & 3 & $(2.27, 2.27)$ & $(2.47, 2.47)$ & N & \cite{ghosh2025newton} \\
		MHHM1 & 1 & 3 & $ 0 $ & $ 1 $ & Y & \cite{huband2006review} \\
		MHHM2 & 2 & 3 & $ (0, 0) $ & $ (1, 1) $ & Y & \cite{huband2006review} \\
		ZLTa & 4 & 4 & $ (-1000, \ldots, -1000) $ & $ (1000, \ldots, 1000) $ & Y & \cite{huband2006review} \\
		ZLTb & 6 & 6 & $ (-1000, \ldots, -1000) $ & $ (1000, \ldots, 1000) $ & Y & \cite{huband2006review} \\
		ZLTc & 8 & 8 & $ (-1000, \ldots, -1000) $ & $ (1000, \ldots, 1000) $ & Y & \cite{huband2006review} \\
		ZLTd & 10 & 10 & $ (-1000, \ldots, -1000) $ & $ (1000, \ldots, 1000) $ & Y & \cite{huband2006review} \\
		TE8a & 8 & 3 & $(0, \ldots, 0)$ & $(10, \ldots, 10)$ & Y & \cite{anh2025steepest} \\
		TE8b & 10 & 3 & $(0, \ldots, 0)$ & $(10, \ldots, 10)$ & Y & \cite{anh2025steepest} \\
		TE8c & 12 & 3 & $(0, \ldots, 0)$ & $(10, \ldots, 10)$ & Y &  \cite{anh2025steepest} \\
		JOS1a & 50 & 2 & $(-2, \ldots, -2)$ & $(2, \ldots, 2)$ & Y & \cite{jin2001dynamic} \\
		JOS1b & 100 & 2 & $(-2, \ldots, -2)$ & $(2, \ldots, 2)$ & Y & \cite{jin2001dynamic} \\
		JOS1c & 500 & 2 & $(-2, \ldots, -2)$ & $(2, \ldots, 2)$ & Y & \cite{jin2001dynamic} \\
		TOI4a & 40 & 2 & $(-2, \ldots, -2)$ & $(3, \ldots, 3)$ & Y & \cite{yin2025multiobjective} \\
		TOI4b & 100 & 2 & $(-2, \ldots, -2)$ & $(3, \ldots, 3)$ & Y & \cite{yin2025multiobjective} \\
		TOI4c & 500 & 2 & $(-2, \ldots, -2)$ & $(3, \ldots, 3)$ & Y & \cite{yin2025multiobjective} \\
		M-MAN1a & 50 & 2 & $ (-10, \ldots, -10) $ & $ (10, \ldots, 10) $ & Y & \cite{lapucci2023limited} \\
		M-MAN1b & 100 & 2 & $ (-10, \ldots, -10) $ & $ (10, \ldots, 10) $ & Y & \cite{lapucci2023limited} \\
		M-MAN1c & 500 & 2 & $ (-10, \ldots, -10) $ & $ (10, \ldots, 10) $ & Y & \cite{lapucci2023limited} \\
		MMR5a & 50 & 2 & $ (-5, \ldots, -5) $ & $ (5, \ldots, 5) $ & N & \cite{miglierina2008box} \\
		MMR5b & 100 & 2 & $ (-5, \ldots, -5) $ & $ (5, \ldots, 5) $ & N & \cite{miglierina2008box} \\
		MMR5c & 500 & 2 & $ (-5, \ldots, -5) $ & $ (5, \ldots, 5) $ & N & \cite{miglierina2008box} \\
		QV1a & 50 & 2 & $ (-5.12, \ldots, -5.12) $ & $ (5.12, \ldots, 5.12) $ & N & \cite{huband2006review} \\
		QV1b & 100 & 2 & $ (-5.12, \ldots, -5.12) $ & $ (5.12, \ldots, 5.12) $ & N & \cite{huband2006review} \\
		QV1c & 500 & 2 & $ (-5.12, \ldots, -5.12) $ & $ (5.12, \ldots, 5.12) $ & N & \cite{huband2006review} \\
		ZLT1a & 50 & 3 & $ (-1000, \ldots, -1000) $ & $ (1000, \ldots, 1000) $ & Y & \cite{huband2006review} \\
		ZLT1b & 100 & 3 & $ (-1000, \ldots, -1000) $ & $ (1000, \ldots, 1000) $ & Y & \cite{huband2006review} \\
		ZLT1c & 500 & 3 & $ (-1000, \ldots, -1000) $ & $ (1000, \ldots, 1000) $ & Y & \cite{huband2006review} \\
		\bottomrule
	\end{tabular}
	
\end{table}

\begin{table}[H]
	\centering
	\caption{The description of  higher-dimensional test problems used in numerical experiments.}
	\label{tab2:problem_description}
	\begin{tabular}{@{}l c c l l c l@{}}
		\toprule
		\textbf{Problem} & \textbf{n} & \textbf{m} & \textbf{\(x_L\)} & \textbf{\(x_U\)} & \textbf{Convex} & \textbf{References} \\
		\midrule
		JOS1d & 1000 & 2 & $(-2, \ldots, -2)$ & $(2, \ldots, 2)$ & Y & \cite{jin2001dynamic} \\
		JOS1e & 2000 & 2 & $(-2, \ldots, -2)$ & $(2, \ldots, 2)$ & Y & \cite{jin2001dynamic} \\
		JOS1f & 5000 & 2 & $(-2, \ldots, -2)$ & $(2, \ldots, 2)$ & Y & \cite{jin2001dynamic} \\
		JOS1g & 10000 & 2 & $(-2, \ldots, -2)$ & $(2, \ldots, 2)$ & Y & \cite{jin2001dynamic} \\
		TOI4d & 1000 & 2 & $(-2, \ldots, -2)$ & $(3, \ldots, 3)$ & Y & \cite{yin2025multiobjective} \\
		TOI4e & 2000 & 2 & $(-2, \ldots, -2)$ & $(3, \ldots, 3)$ & Y & \cite{yin2025multiobjective} \\
		TOI4f & 5000 & 2 & $(-2, \ldots, -2)$ & $(3, \ldots, 3)$ & Y & \cite{yin2025multiobjective} \\
		TOI4g & 10000 & 2 & $(-2, \ldots, -2)$ & $(3, \ldots, 3)$ & Y & \cite{yin2025multiobjective} \\
		M-MAN1d & 1000 & 2 & $ (-10, \ldots, -10) $ & $ (10, \ldots, 10) $ & Y & \cite{lapucci2023limited} \\
		M-MAN1e & 2000 & 2 & $ (-10, \ldots, -10) $ & $ (10, \ldots, 10) $ & Y &
		 \cite{lapucci2023limited} \\
		 M-MAN1f & 5000 & 2 & $ (-10, \ldots, -10) $ & $ (10, \ldots, 10) $ & Y &
		 \cite{lapucci2023limited} \\
		MMR5d & 1000 & 2 & $ (-5, \ldots, -5) $ & $ (5, \ldots, 5) $ & N & \cite{miglierina2008box} \\
		MMR5e & 2000 & 2 & $ (-5, \ldots, -5) $ & $ (5, \ldots, 5) $ & N & \cite{miglierina2008box} \\
		MMR5f & 5000 & 2 & $ (-5, \ldots, -5) $ & $ (5, \ldots, 5) $ & N & \cite{miglierina2008box} \\
		QV1d & 1000 & 2 & $ (-5.12, \ldots, -5.12) $ & $ (5.12, \ldots, 5.12) $ & N & \cite{huband2006review} \\
		QV1e & 2000 & 2 & $ (-5.12, \ldots, -5.12) $ & $ (5.12, \ldots, 5.12) $ & N & \cite{huband2006review} \\
		QV1f & 5000 & 2 & $ (-5.12, \ldots, -5.12) $ & $ (5.12, \ldots, 5.12) $ & N & \cite{huband2006review} \\
		ZLT1d & 1000 & 3 & $ (-1000, \ldots, -1000) $ & $ (1000, \ldots, 1000) $ & Y & \cite{huband2006review} \\
		ZLT1e & 2000 & 3 & $ (-1000, \ldots, -1000) $ & $ (1000, \ldots, 1000) $ & Y & \cite{huband2006review} \\
		ZLT1f & 5000 & 3 & $ (-1000, \ldots, -1000) $ & $ (1000, \ldots, 1000) $ & Y & \cite{huband2006review} \\
		ZLT1g & 10000 & 3 & $ (-1000, \ldots, -1000) $ & $ (1000, \ldots, 1000) $ & Y & \cite{huband2006review} \\
		\bottomrule
	\end{tabular}
\end{table}

\begin{landscape}  
	\begin{table}[p]  
		\centering
		\caption{Number of average CPU time (time(ms)), average iterations (iter), number of average function evaluations (feval), and number of failure points (NF) of M-BFGS, BB-DQN, and D-QN.}
		\label{tab1:comparison_landscape}
		\begin{tabular}{lcccccccccccc}
			\toprule
			\multirow{2}{*}{Problem} & \multicolumn{4}{c}{M-BFGS} & \multicolumn{4}{c}{BB-DQN} & \multicolumn{4}{c}{D-QN} \\
			\cmidrule(lr){2-5} \cmidrule(lr){6-9} \cmidrule(lr){10-13}
			& time & iter & feval & NF & time & iter & feval & NF & time & iter & feval & NF \\
			\midrule
			SLCDT1 & 0.15 & \textbf{4.27} & \textbf{11.42} & 0 & \textbf{0.13} & 4.70 & 12.10 & 0 & 0.86 & 23.25 & 90.95 & 0 \\
			PNR    & 0.36 & 9.69 & 36.03 & 0 & \textbf{0.30} & 8.73 & 31.56 & 0 & 0.32 & \textbf{8.21} & \textbf{28.42} & 0 \\
			MOP2   & 0.24 & 4.77 & \textbf{17.34} & 0 & \textbf{0.18} & \textbf{4.74} & 17.44 & 0 & 0.24 & 5.48 & 19.89 & 0 \\
			MOP5   & 7.40 & 7.65 & 27.83 & 1 & \textbf{5.63} & 7.42 & 25.62 & 0 & 6.14 & \textbf{7.37} & \textbf{24.76} & 0 \\
			MOP7   & 0.07 & 2.00 & 3.00 & 0 & 0.07 & 2.00 & 3.00 & 0 & 0.07 & 2.00 & 3.00 & 0 \\
			Far1   & \textbf{32.78} & 2.67 & \textbf{6.42} & 0 & 33.40 & \textbf{2.66} & 6.47 & 0 & 33.86 & \textbf{2.66} & 6.49 & 0 \\
			KW2    & 1.44 & 11.95 & 54.33 & 8 & 0.98 & 10.92 & 44.67 & 0 & \textbf{0.55} & \textbf{10.36} & \textbf{34.32} & 0 \\
			FF1    & 0.34 & \textbf{7.95} & \textbf{31.06} & 0 & \textbf{0.3} & 10.51 & 43.24 & 0 & 0.36 & 9.61 & 33.70 & 0 \\
			Deb    & 4.48 & 94.89 & 754.75 & 11 & \textbf{2.46} & \textbf{68.34} & \textbf{509.97} & 9 & 7.82 & 131.67 & 806.48 & 7 \\
			DD     & 6.13 & 74.39 & 221.40 & 7 & 5.72 & 73.28 & \textbf{217.93} & 7 & \textbf{5.25} & \textbf{65.41} & 255.65 & 6 \\
			BK1    & 0.07 & 2.00 & 4.00 & 0 & \textbf{0.06} & 2.00 & 4.00 & 0 & 0.08 & 2.00 & 4.00 & 0 \\
			WIT    & 0.34 & 8.96 & 31.35 & 0 & \textbf{0.25} & \textbf{7.91} & \textbf{27.06} & 0 & 0.34 & 8.68 & 27.51 & 0 \\
			NT2a   & 0.34 & \textbf{5.15} & \textbf{14.45} & 0 & 0.34 & 5.25 & 15.82 & 0 & 0.90 & 12.41 & 36.41 & 0 \\
			NT2b   & 1.27 & \textbf{5.79} & \textbf{21.11} & 0 & \textbf{0.81} & 5.91 & 23.45 & 0 & 2.63 & 22.14 & 71.81 & 0 \\
			NT2c   & 2.17 & \textbf{5.28} & \textbf{22.79} & 0 & \textbf{1.28} & 5.86 & 26.78 & 0 & 7.34 & 24.83 & 85.30 & 0 \\
			ZKG7   & 0.27 & 1.42 & 3.61 & 3 & \textbf{0.14} & \textbf{1.38} & \textbf{3.34} & 3 & 14.33 & 30.99 & 90.05 & 3 \\
			MHHM1  & 0.05 & 1.9 & 3.6 & 0 & \textbf{0.04} & 1.9 & 3.6 & 0 & 0.04 & 1.9 & 3.6 & 0 \\
			MHHM2  & 259.43 & \textbf{142.23} & 6022.84 & 35 & \textbf{234.52} & 142.38 & \textbf{5819.34} & 34 & 673.51 & 341.66 & 16000.56 & 34 \\
			ZLTa  & 1.42  & 2.00 & 4.00 & 0 & \textbf{1.34}  & 2.00 & 4.00 & 0 & 1.40  & 2.00 & 4.00 & 0 \\
			ZLTb  & 1.29  & 2.00 & 4.00 & 0 & \textbf{1.24}  & 2.00 & 4.00 & 0 & 1.26  & 2.00 & 4.00 & 0 \\
			ZLTc  & 1.26  & 2.00 & 4.00 & 0 & \textbf{1.15}  & 2.00 & 4.00 & 0 & 1.44  & 2.00 & 4.00 & 0 \\
			\bottomrule
		\end{tabular}
	\end{table}
\end{landscape}

\begin{landscape}  
	\begin{table}[p]  
		\caption*{continue:}
		\centering
		\begin{tabular}{lcccccccccccc}
			\toprule
			\multirow{2}{*}{Problem} & \multicolumn{4}{c}{M-BFGS} & \multicolumn{4}{c}{BB-DQN} & \multicolumn{4}{c}{D-QN} \\
			\cmidrule(lr){2-5} \cmidrule(lr){6-9} \cmidrule(lr){10-13}
			& time & iter & feval & NF & time & iter & feval & NF & time & iter & feval & NF \\
			\midrule
			ZLTd  & 1.46  & 2.00 & 4.00 & 0 & \textbf{1.43}  & 2.00 & 4.00 & 0 & 1.70  & 2.00 & 4.00 & 0 \\
			TE8a   & 22.87 & 13.05 & 460.16 & 9 & \textbf{22.28} & \textbf{12.74} & \textbf{377.87} & 8 & 69.92 & 33.12 & 1117.48 & 4 \\
			TE8b   & 8.73 & 5.72 & 167.31 & 3 & \textbf{0.89} & \textbf{2.34} & \textbf{10.97} & 3 & 66.27 & 31.97 & 1140.37 & 3 \\
			TE8c   & 25.81 & 12.28 & 455.68 & 2 & \textbf{1.09} & \textbf{2.24} & \textbf{7.11} & 2 & 44.78 & 21.98 & 591.84 & 2 \\
			JOS1a & 6.93 & 7.16 & 32.38 & 0 & \textbf{1.40} & \textbf{3.00} & \textbf{10.00} & 0 & 11.91 & 17.95 & 103.79 & 0 \\
			JOS1b & 8.60 & 7.37 & 36.67 & 0 & \textbf{1.77} & \textbf{3.00} & \textbf{11.00} & 0 & 9.74 & 9.50 & 71.82 & 0 \\
			JOS1c & 58.44 & 4.87 & 25.86 & 0 & \textbf{9.41} & \textbf{3.00} & \textbf{12.00} & 0 & 12.61 & 3.00 & 18.00 & 0 \\
			TOI4a & 3.67 & 11.75 & 42.59 & 0 & \textbf{1.34} & \textbf{6.26} & \textbf{16.24} & 0 & 2.60 & 14.47 & 40.62 & 0 \\
			TOI4b & 6.98 & 12.81 & 48.31 & 0 & \textbf{1.48} & \textbf{6.46} & \textbf{17.12} & 0 & 9.22 & 20.09 & 57.32 & 0 \\
			TOI4c & 232.38 & 18.08 & 77.23 & 0 & \textbf{5.04} & \textbf{5.15} & \textbf{12.49} & 0 & 190.17 & 40.66 & 118.98 & 0 \\
			M-MAN1a & 18.05 & 46.25 & 179.48 & 0 & \textbf{8.49} & \textbf{28.70} & \textbf{120.46} & 0 & 10.24 & 30.39 & 171.67 & 0 \\
			M-MAN1b & 63.54 & 86.65 & 380.01 & 0 & \textbf{13.33} & \textbf{45.42} & \textbf{193.19} & 0 & 39.79 & 58.87 & 363.60 & 0 \\
			M-MAN1c & 7314.67 & 502.44 & 2436.07 & 0 & \textbf{382.53} & \textbf{115.26} & \textbf{494.55} & 0 & 2805.84 & 464.10 & 3344.88 & 0 \\
			MMR5a & 121.70 & 286.15 & 1963.73 & 0 & 44.70 & \textbf{159.93} & 1276.50 & 0 & 57.11 & 203.50 & \textbf{1008.61} & 0 \\
			MMR5b & 204.87 & 264.70 & 1793.40 & 0 & \textbf{54.97} & \textbf{134.71} & 1033.13 & 0 & 107.99 & 156.32 & \textbf{924.07} & 0 \\
			MMR5c & 2420.68 & 164.95 & 1121.04 & 0 & \textbf{323.85} & 86.87 & \textbf{588.99} & 0 & 366.07 & \textbf{62.23} & 619.35 & 0 \\
			QV1a & 107.55 & 269.65 & 1814.95 & 0 & \textbf{43.94} & \textbf{154.49} & 1226.65 & 0 & 52.82 & 194.76 & \textbf{975.76} & 0 \\
			QV1b & 215.73 & 262.17 & 1756.66 & 0 & \textbf{56.07} & \textbf{132.41} & 1011.41 & 0 & 116.59 & 161.59 & \textbf{948.97} & 0 \\
			QV1c & 1948.21 & 130.45 & 871.10 & 0 & \textbf{282.96} & 85.01 & 579.95 & 0 & 291.84 & \textbf{45.31} & \textbf{394.17} & 0 \\
			ZLT1a & 1.39 & 2.00 & 4.00 & 0 & 0.97 & 2.00 & 4.00 & 0 & 0.86 & 2.00 & 4.00 & 0 \\
			ZLT1b & 1.67 & 2.00 & 4.00 & 0 & \textbf{0.97} & 2.00 & 4.00 & 0 & 1.30 & 2.00 & 4.00 & 0 \\
			ZLT1c & 14.56 & 2.00 & 4.00 & 0 & \textbf{5.39} & 2.00 & 4.00 & 0 & 11.64 & 2.00 & 4.00 & 0 \\
			\bottomrule
		\end{tabular}
	\end{table}
\end{landscape}

\begin{landscape}  
	\begin{table}[p]  
		\centering
		\caption{Number of average CPU time (time(ms)), average iterations (iter), number of average function evaluations (feval), and number of failure points (NF) of M-BFGS, M-BB, and D-QN.}
		\label{tab3:comparison_landscape}
		\begin{tabular}{lcccccccccccc}
			\toprule
			\multirow{2}{*}{Problem} & \multicolumn{4}{c}{M-BFGS} & \multicolumn{4}{c}{BB-DQN} & \multicolumn{4}{c}{D-QN} \\
			\cmidrule(lr){2-5} \cmidrule(lr){6-9} \cmidrule(lr){10-13}
			& time & iter & feval & NF & time & iter & feval & NF & time & iter & feval & NF \\
			\midrule
			JOS1d & 286.39 & 4.46 & 28.41 & 0 & \textbf{25.94} & \textbf{3.00} & \textbf{13.00} & 0 & 97.77 & 6.00 & 50.00 & 0 \\
			JOS1e & 3093.38 & 6.78 & 46.93 & 0 & \textbf{86.44} & \textbf{3.00} & \textbf{14.00} & 0 & 707.39 & 11.00 & 110.00 & 0 \\
			JOS1f & 36587.73 & 6.22 & 47.93 & 0 & \textbf{561.63} & \textbf{3.00} & \textbf{15.00} & 0 & 3949.47 & 11.00 & 120.00 & 0 \\
			JOS1g & 334930.06 & 7.57 & 63.51 & 0 & \textbf{2521.35} & \textbf{3.00} & \textbf{16.00} & 0 & 104830.2 & 69.0 & 883.6 & 0 \\
			TOI4d & 1473.33 & 20.07 & 88.74 & 0 & \textbf{23.67} & \textbf{5.09} & \textbf{12.27} & 0 & 983.26 & 56.27 & 165.80 & 0 \\
			TOI4e & 10347.36 & 22.52 & 103.31 & 0 & \textbf{85.01} & \textbf{5.18} & \textbf{12.53} & 0 & 5036.19 & 77.76 & 230.28 & 0 \\
			TOI4f & 399900.08 & 26.63 & 137.73 & 1 & \textbf{731.54} & \textbf{7.10} & \textbf{19.90} & 0 & 42917.75 & 118.64 & 352.92 & 0 \\
			TOI4g & F & F & F & F & \textbf{2917.31} & \textbf{7.22} & \textbf{20.41} & 0 & 251654.30 & 163.97 & 488.91 & 0 \\
			M-MAN1d & 85460.87 & 1090.43 & 5350.82 & 0 & \textbf{1656.84} & \textbf{172.07} & \textbf{746.53} & 0 & 16672.19 & 808.46 & 5994.60 & 0 \\
			M-MAN1e & 950696.76 & 1999.92 & 10072.90 & 198 & \textbf{8257.52} & \textbf{263.59} & \textbf{1154.16} & 0 & 115615.42 & 1667.24 & 12879.54 & 1 \\
			MMR5d & 21763.64 & 278.14 & 1958.61 & 0 & \textbf{1287.94} & \textbf{125.18} & \textbf{917.53} & 0 & 4097.67 & 199.34 & 2476.64 & 2 \\
			MMR5e & 140782.58 & 298.76 & 2097.31 & 0 & \textbf{6199.99} & \textbf{185.99} & \textbf{1388.12} & 0 & 20781.26 & 273.44 & 6751.14 & 10 \\
			MMR5f & F & F & F & F & \textbf{788923.47} & 214.46 & 1615.01 & 0 & F & F & F & F \\
			QV1d & 22566.25 & 290.03 & 2006.56 & 0 & \textbf{1445.72} & \textbf{140.72} & \textbf{1035.13} & 0 & 3695.81 & 170.00 & 1991.63 & 2 \\
			QV1e & 127313.32 & 266.25 & 1797.62 & 0 & \textbf{2043.19} & \textbf{63.70} & \textbf{412.28} & 0 & 14493.26 & 198.23 & 4067.05 & 4 \\
			QV1f & F & F & F & 0 & \textbf{4432.07} & \textbf{43.00} & \textbf{248.90} & 0 & 98232.05 & 232.95 & 6508.30 & 14 \\
			ZLT1d & 86.64 & 2.00 & 4.00 & 0 & \textbf{17.79} & 2.00 & 4.00 & 0 & 31.13 & 2.00 & 4.00 & 0 \\
			ZLT1e & 494.12 & 2.00 & 4.00 & 0 & \textbf{41.19} & 2.00 & 4.00 & 0 & 88.13 & 2.00 & 4.00 & 0 \\
			ZLT2f & 7685.22 & 2.00 & 4.00 & 0 & \textbf{373.87} & 2.00 & 4.00 & 0 & 511.11 & 2.00 & 4.00 & 0 \\
			ZLT3g & 57942.47 & 2.00 & 4.00 & 0 & \textbf{1623.82} & 2.00 & 4.00 & 0 & 2121.53 & 2.00 & 4.00 & 0 \\
			\bottomrule
		\end{tabular}
	\end{table}
\end{landscape}

The experimental results in Table \ref{tab1:comparison_landscape} and \ref{tab3:comparison_landscape} demonstrate that the proposed BB-DQN method consistently outperforms or matches benchmark algorithms across most test instances, with its advantages becoming particularly pronounced on large-scale problems.

BB-DQN method exhibits superior efficiency in terms of CPU time for the vast majority of problems (Figure \ref{CPU}). While maintaining comparable runtimes on medium and small-scale instances (Table~\ref{tab1:comparison_landscape}), its decisive advantage emerges in high-dimensional scenarios (Table~\ref{tab3:comparison_landscape}). For instance, on the 10,000-variable problem \texttt{JOS1g}, BB-DQN method required only 2521.35 ms, representing a dramatic speedup compared to the 334,930.06 ms for M-BFGS method and the 104,830.2 ms for D-QN. Furthermore, BB-DQN method often converges in fewer iterations and with fewer function evaluations, indicating that its simplified structure does not compromise search efficacy.

Performance profile analyses based on the Purity, $\Gamma$ and $\Delta$ metrics confirm that BB-DQN method's computational gains are not achieved at the expense of solution quality. As shown in Figure~\ref{P} and \ref{GD}, BB-DQN method generates Pareto front approximations that are highly competitive with, and often superior to, those obtained by M-BFGS and D-QN method. It achieves high Purity values, indicating effective convergence to the true Pareto set, while maintaining low $\Gamma$ and $\Delta$ values, which demonstrates that the solution sets are well-distributed and uniformly spread across the front.

The visual comparisons clearly demonstrate that BB-DQN method is capable of generating satisfactory approximations of the Pareto front, solidifying its position as a balanced algorithm that excels in both computational efficiency and optimization quality. Figures~\ref{fig9:pareto1} and \ref{fig10:pareto1} visualize the results for representative convex problems, while Figures~\ref{fig1:pareto1} - \ref{fig11:pareto1} extend the comparison to nonconvex instances. Figure \ref{fig12:pareto1} shows the approximated nondominated frontiers for the MHHM2 problem with 2500 starting points.

\begin{figure}[H]
	\centering
	\begin{subfigure}[b]{0.3\textwidth}
		\centering
		\includegraphics[width=0.9\textwidth]{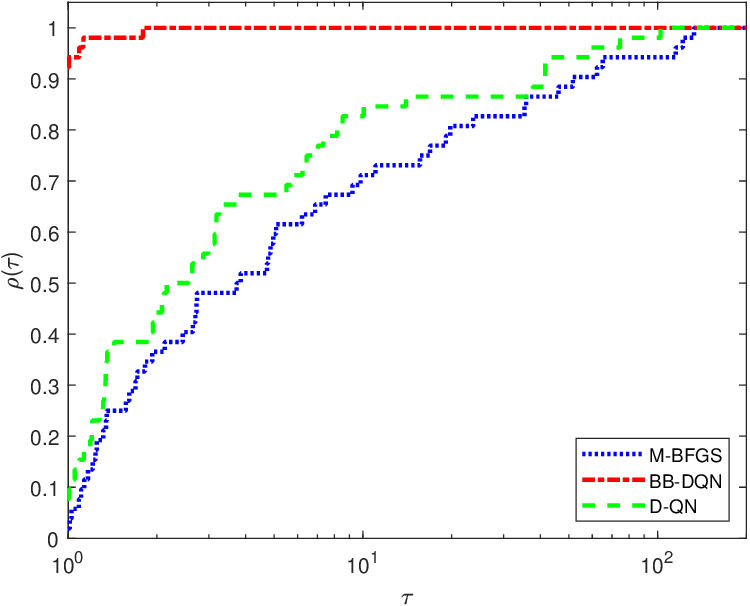}
		\caption{CPU time}
	\end{subfigure}
	\hfill
	\begin{subfigure}[b]{0.3\textwidth}
		\centering
		\includegraphics[width=0.9\textwidth]{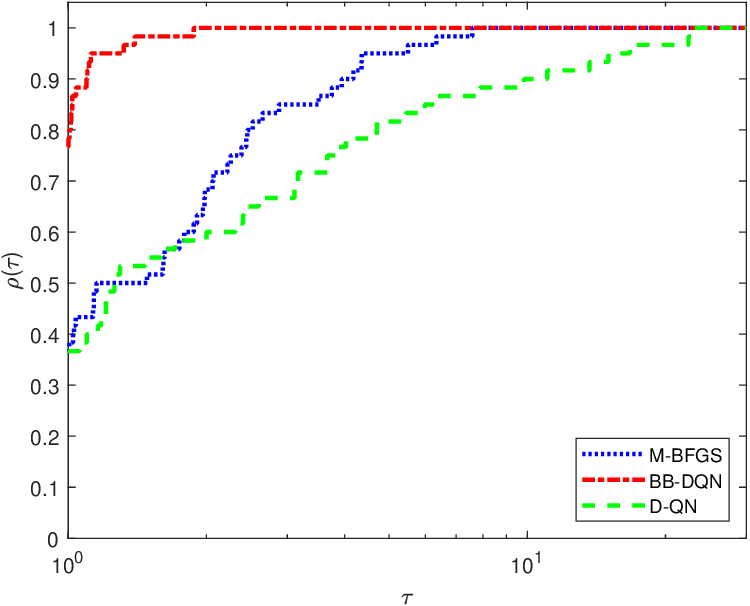}
		\caption{iter}
	\end{subfigure}
	\hfill
	\begin{subfigure}[b]{0.3\textwidth}
		\centering
		\includegraphics[width=0.9\textwidth]{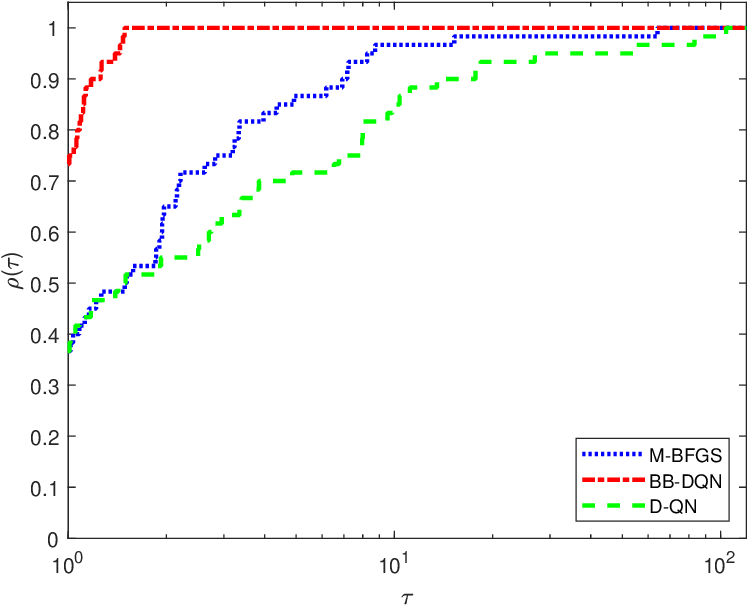}
		\caption{feval}
	\end{subfigure}
	\caption{Performance profiles with CPU time, iter and feval for each test problem.}
	\label{CPU}
\end{figure}

\begin{figure}[H]
	\centering
	\begin{subfigure}[b]{0.44\textwidth}
		\centering
		\includegraphics[width=0.9\textwidth]{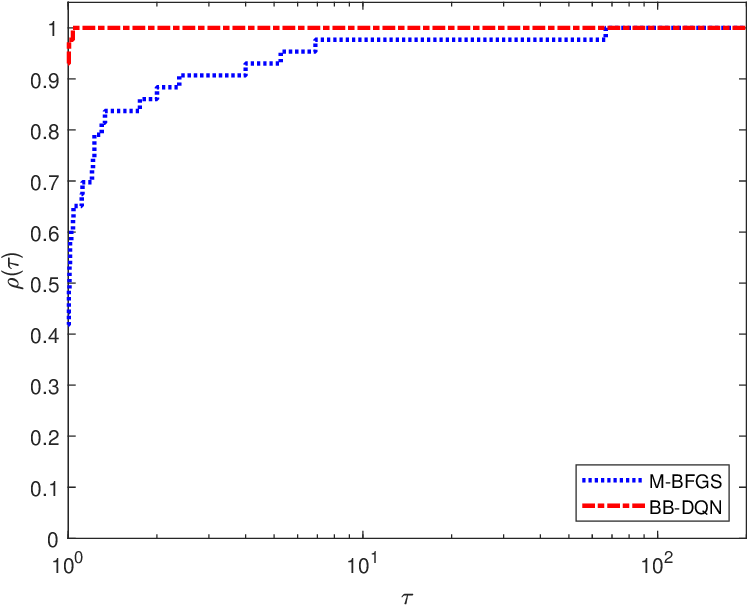}
		\label{fig1:CPU}
	\end{subfigure}
	\hfill
	\begin{subfigure}[b]{0.44\textwidth}
		\centering
		\includegraphics[width=0.9\textwidth]{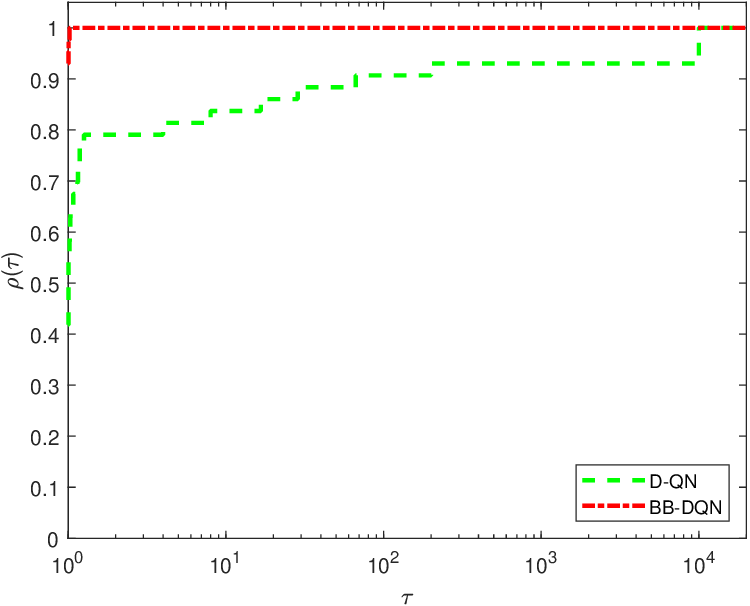}
		\label{fig1:P}
	\end{subfigure}
	\caption{Purity metric performance profiles.}
	\label{P}
\end{figure}

\begin{figure}[H]
	\centering
	\begin{subfigure}[b]{0.44\textwidth}
		\centering
		\includegraphics[width=0.9\textwidth]{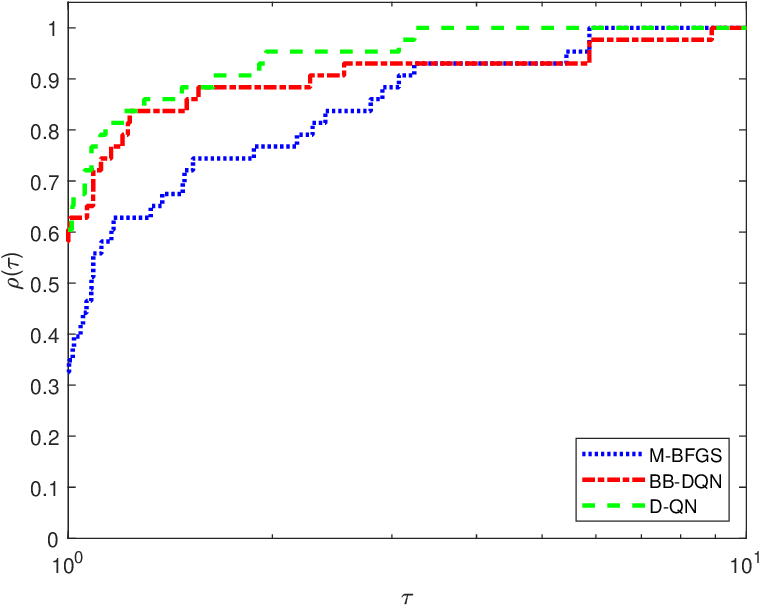}
		\caption{$\Gamma$ metric}
		\label{fig1:G}
	\end{subfigure}
	\hfill
	\begin{subfigure}[b]{0.44\textwidth}
		\centering
		\includegraphics[width=0.9\textwidth]{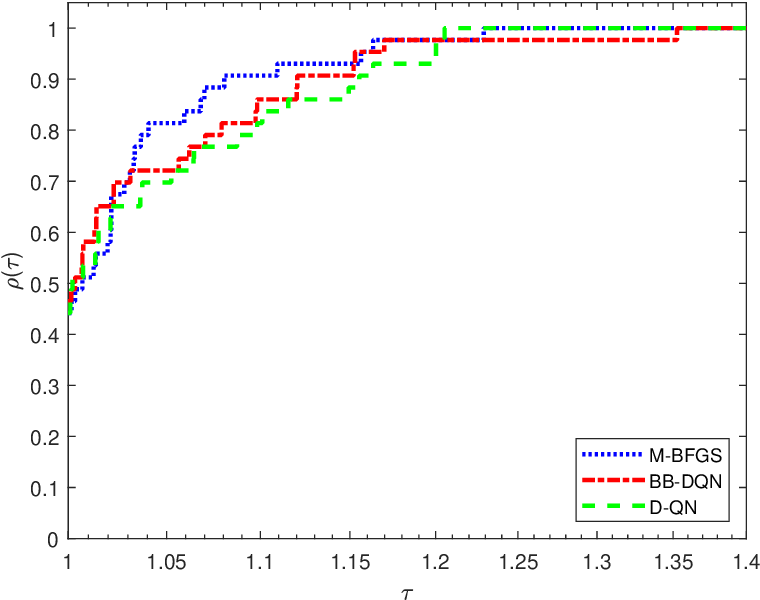}
		\caption{$\Delta$ metric}
		\label{fig1:D}
	\end{subfigure}
	\caption{Spread metric performance profiles: (a) $\Gamma$ metric; (b)$\Delta$ metric.}
	\label{GD}
\end{figure}

\begin{figure}[h!]
	\centering
	
	\begin{subfigure}[b]{0.3\textwidth}
		\centering
		\includegraphics[width=0.9\textwidth]{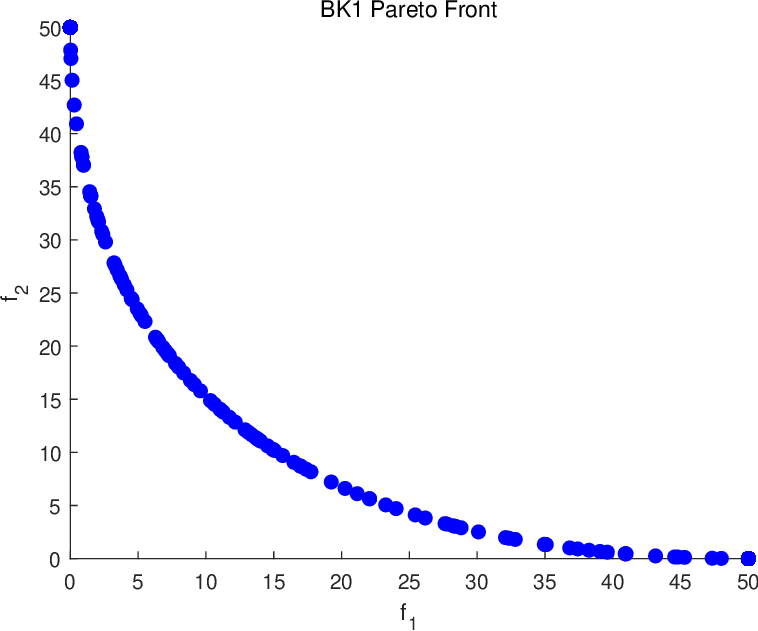}
		\caption{M-BFGS}
		\label{fig9:qnm}
	\end{subfigure}
	\hfill
	\begin{subfigure}[b]{0.3\textwidth}
		\centering
		\includegraphics[width=0.9\textwidth]{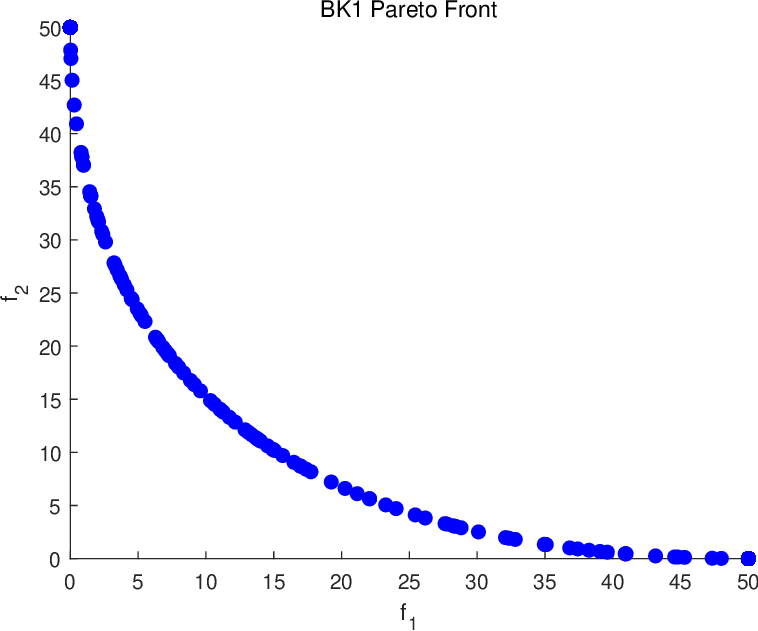}
		\caption{BB-DQN}
		\label{fig9:mbfgs}
	\end{subfigure}
	\hfill
	\begin{subfigure}[b]{0.3\textwidth}
		\centering
		\includegraphics[width=0.9\textwidth]{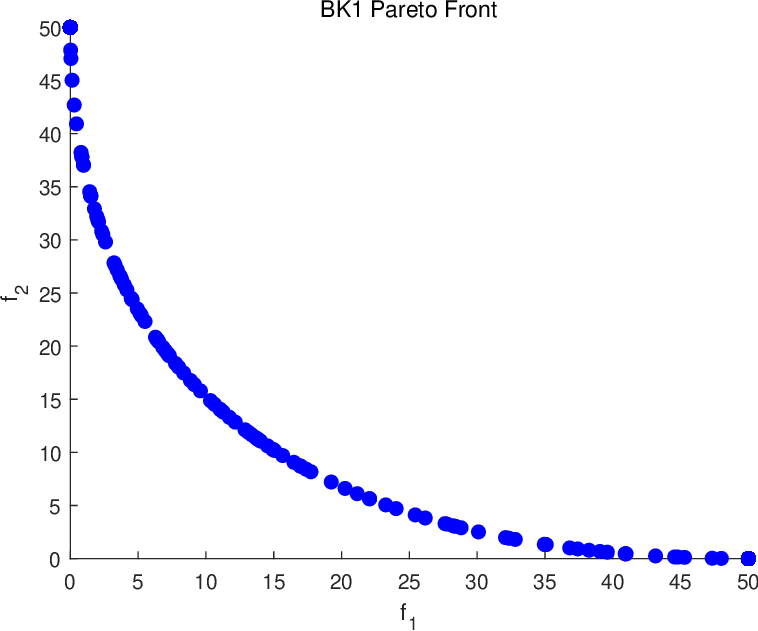}
		\caption{D-QN}
		\label{fig9:DQN}
	\end{subfigure}
	\caption{The approximated nondominated frontiers generated by M-BFGS, BB-DQN and D-QN for the BK1 problem}
	\label{fig9:pareto1}
\end{figure}

\begin{figure}[h!]
	\centering
	
	\begin{subfigure}[b]{0.3\textwidth}
		\centering
		\includegraphics[width=0.9\textwidth]{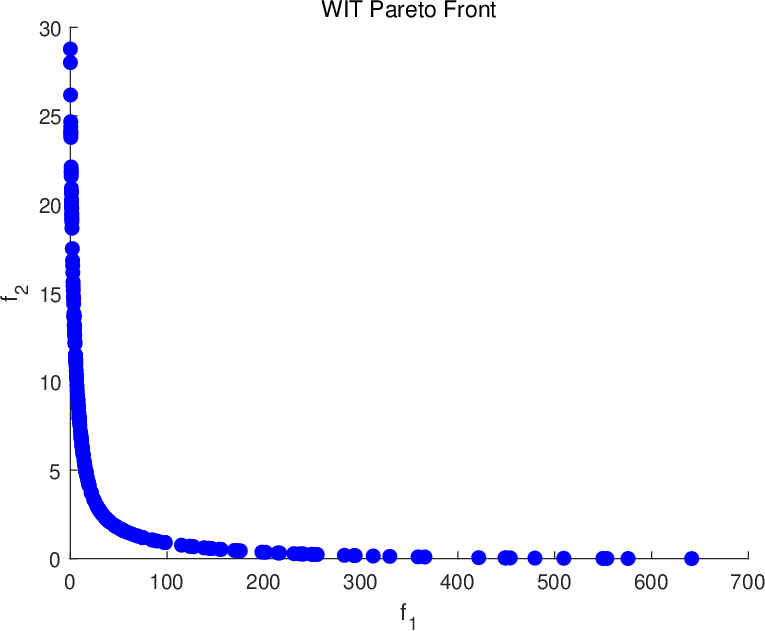}
		\caption{M-BFGS}
		\label{fig10:qnm}
	\end{subfigure}
	\hfill
	\begin{subfigure}[b]{0.3\textwidth}
		\centering
		\includegraphics[width=0.9\textwidth]{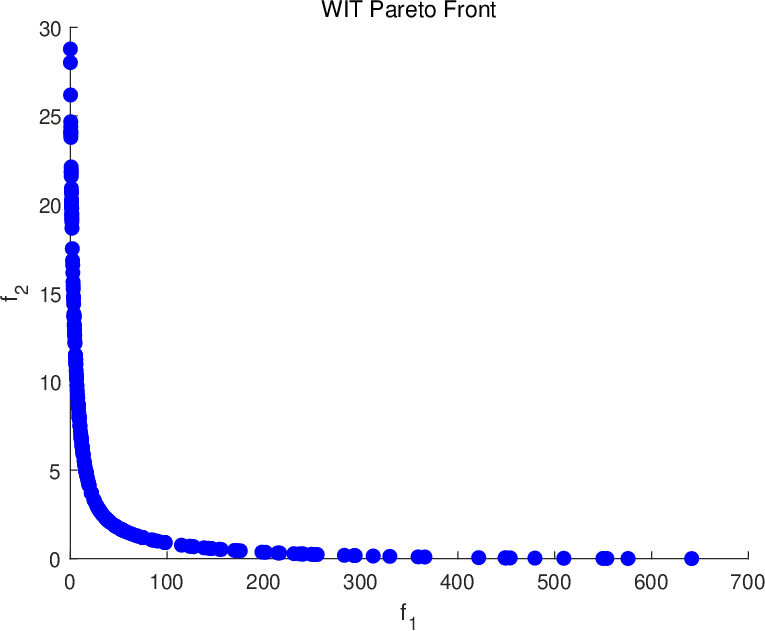}
		\caption{BB-DQN}
		\label{fig10:mbfgs}
	\end{subfigure}
	\hfill
	\begin{subfigure}[b]{0.3\textwidth}
		\centering
		\includegraphics[width=0.9\textwidth]{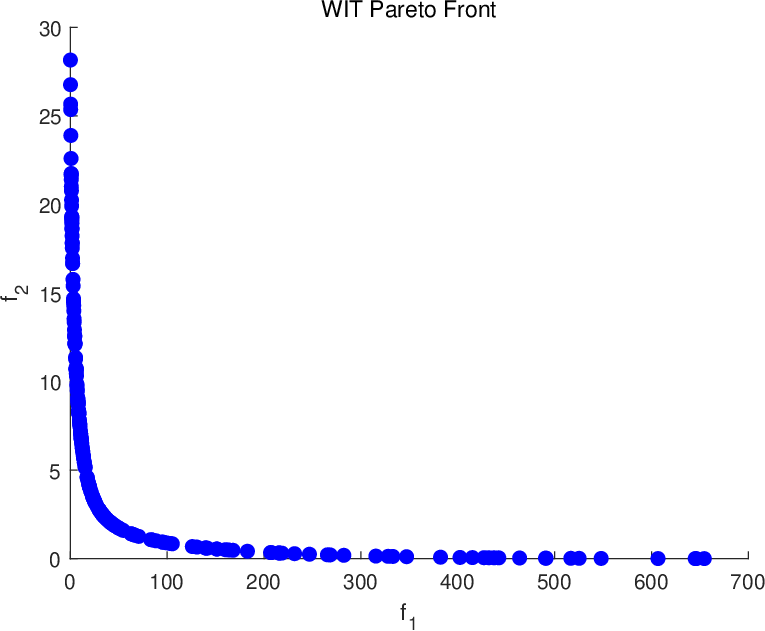}
		\caption{D-QN}
		\label{fig10:DQN}
	\end{subfigure}
	\caption{The approximated nondominated frontiers generated by M-BFGS, BB-DQN and D-QN for the WIT problem}
	\label{fig10:pareto1}
\end{figure}

\begin{figure}[h!]
	\centering
	\begin{subfigure}[b]{0.3\textwidth}
		\centering
		\includegraphics[width=0.9\textwidth]{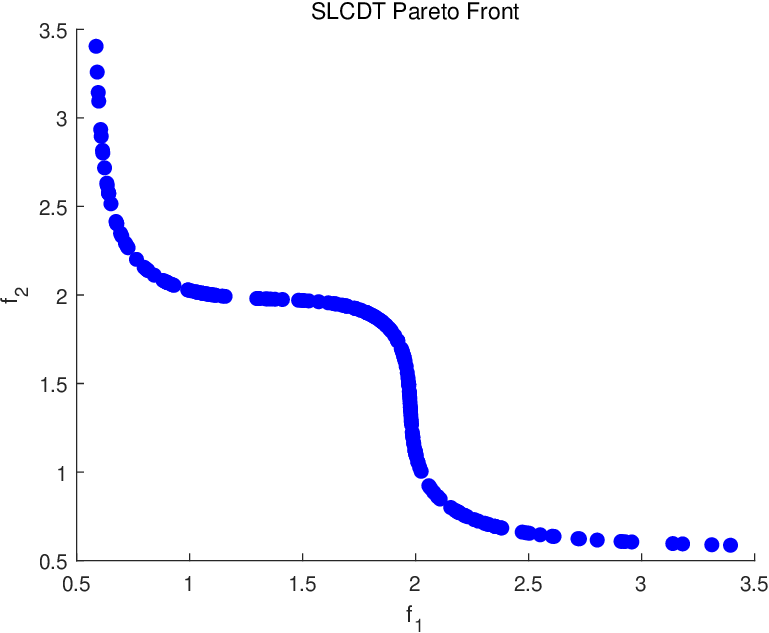}
		\caption{M-BFGS}
		\label{fig1:qnm}
	\end{subfigure}
	\hfill
	\begin{subfigure}[b]{0.3\textwidth}
		\centering
		\includegraphics[width=0.9\textwidth]{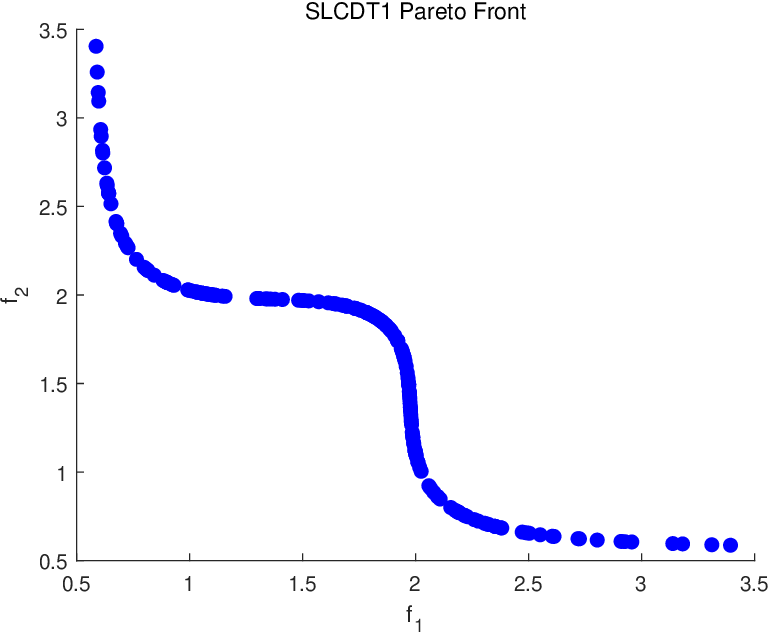}
		\caption{BB-DQN}
		\label{fig1:mbfgs}
	\end{subfigure}
	\hfill
	\begin{subfigure}[b]{0.3\textwidth}
		\centering
		\includegraphics[width=0.9\textwidth]{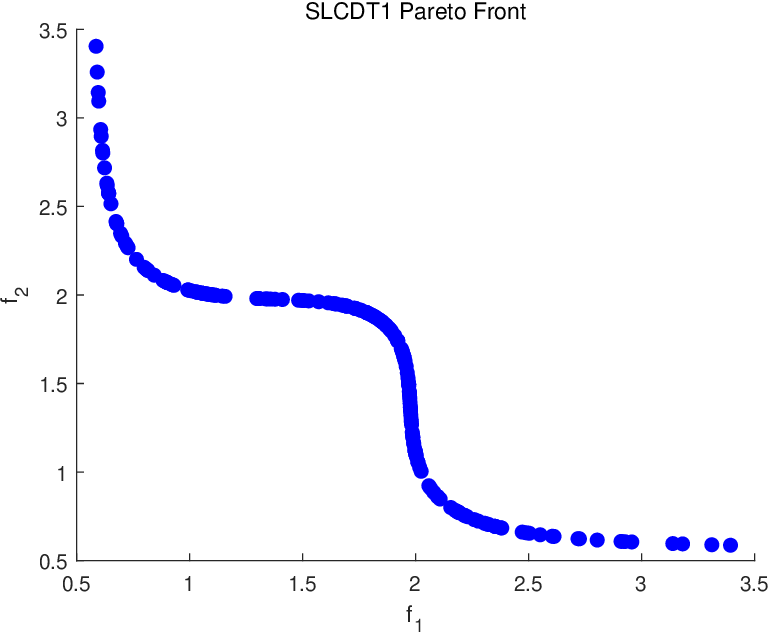}
		\caption{D-QN}
		\label{fig1:DQN}
	\end{subfigure}
	\caption{The approximated nondominated frontiers generated by M-BFGS, BB-DQN and D-QN for the SLCDT1 problem}
	\label{fig1:pareto1}
\end{figure}

\begin{figure}[H]
	\centering
	
	\begin{subfigure}[b]{0.3\textwidth}
		\centering
		\includegraphics[width=0.9\textwidth]{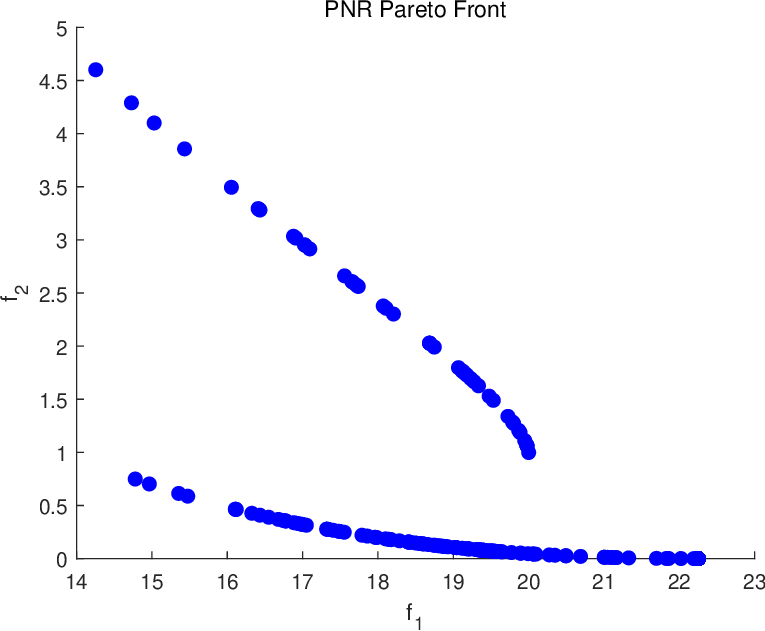}
		\caption{M-BFGS}
		\label{fig2:qnm}
	\end{subfigure}
	\hfill
	\begin{subfigure}[b]{0.3\textwidth}
		\centering
		\includegraphics[width=0.9\textwidth]{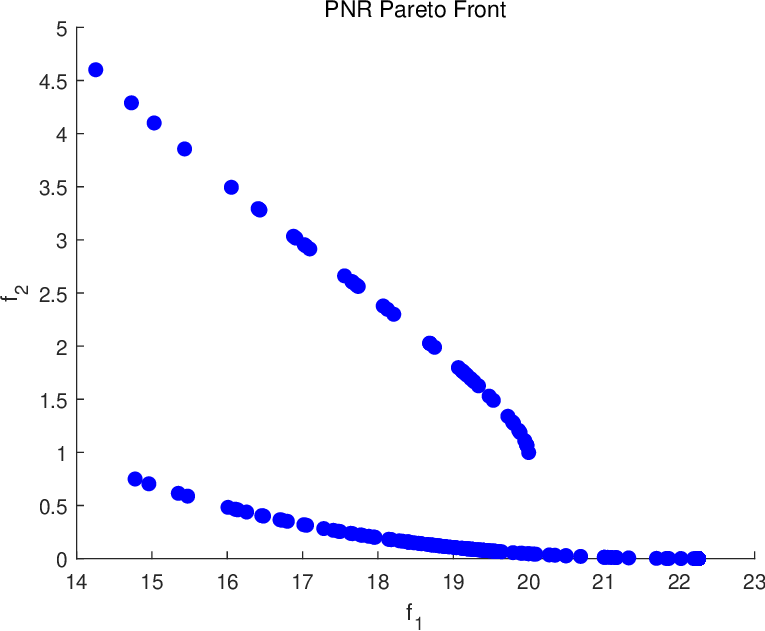}
		\caption{BB-DQN}
		\label{fig2:mbfgs}
	\end{subfigure}
	\hfill
	\begin{subfigure}[b]{0.3\textwidth}
		\centering
		\includegraphics[width=0.9\textwidth]{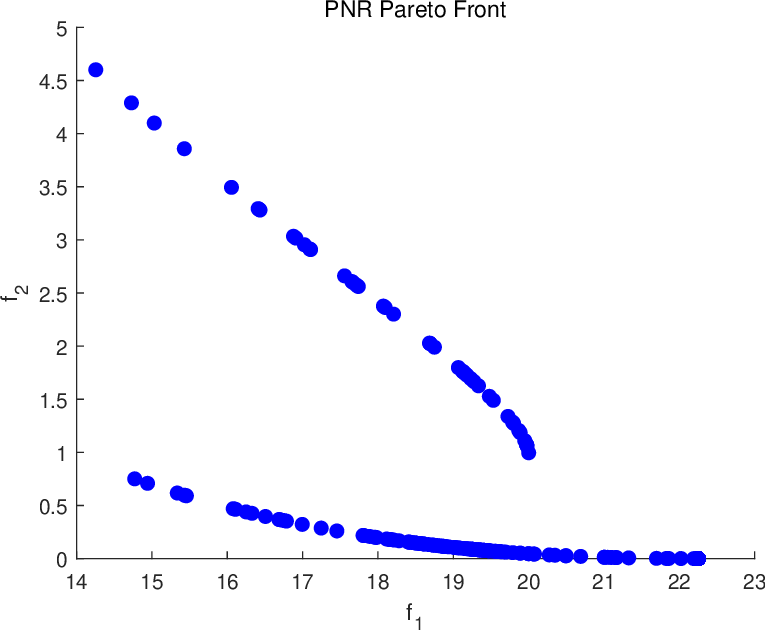}
		\caption{D-QN}
		\label{fig2:DQN}
	\end{subfigure}
	\caption{The approximated nondominated frontiers generated by M-BFGS, BB-DQN and D-QN for the PNR problem}
	\label{fig2:pareto1}
\end{figure}

\begin{figure}[h!]
	\centering
	
	\begin{subfigure}[b]{0.3\textwidth}
		\centering
		\includegraphics[width=0.9\textwidth]{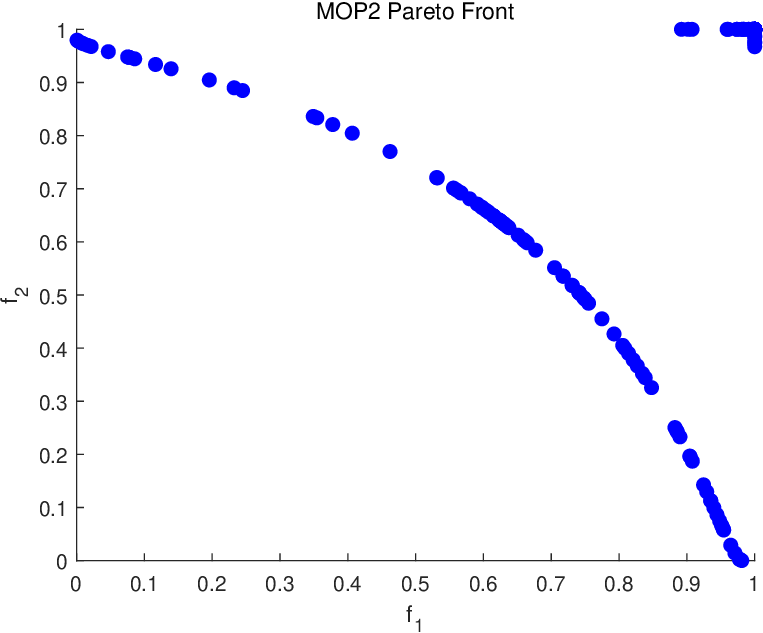}
		\caption{M-BFGS}
		\label{fig3:qnm}
	\end{subfigure}
	\hfill
	\begin{subfigure}[b]{0.3\textwidth}
		\centering
		\includegraphics[width=0.9\textwidth]{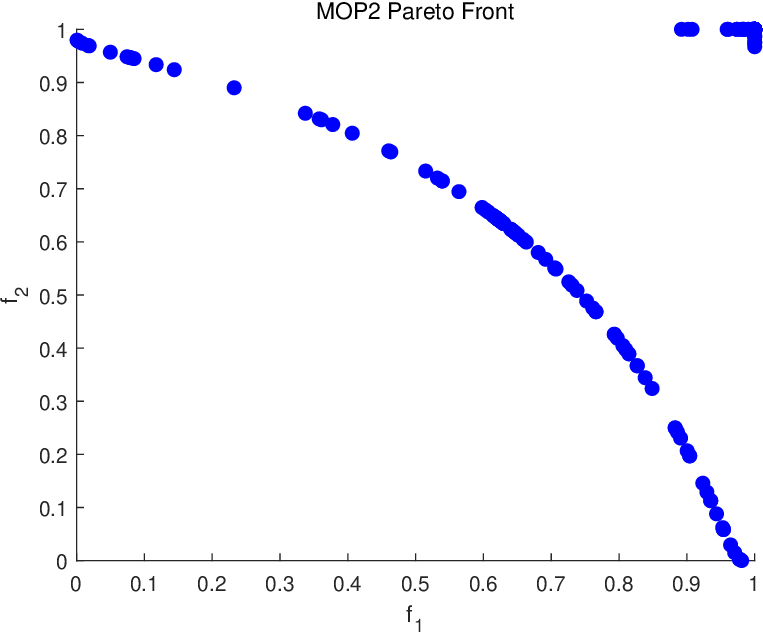}
		\caption{BB-DQN}
		\label{fig3:mbfgs}
	\end{subfigure}
	\hfill
	\begin{subfigure}[b]{0.3\textwidth}
		\centering
		\includegraphics[width=0.9\textwidth]{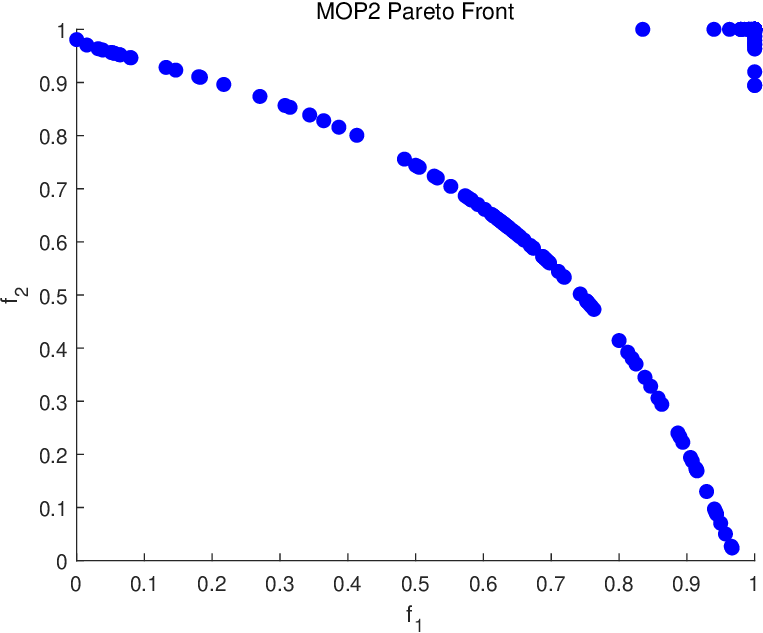}
		\caption{D-QN}
		\label{fig3:DQN}
	\end{subfigure}
	\caption{The approximated nondominated frontiers generated by M-BFGS, BB-DQN and D-QN for the MOP2 problem}
	\label{fig3:pareto1}
\end{figure}

\begin{figure}[h!]
	\centering
	
	\begin{subfigure}[b]{0.3\textwidth}
		\centering
		\includegraphics[width=0.9\textwidth]{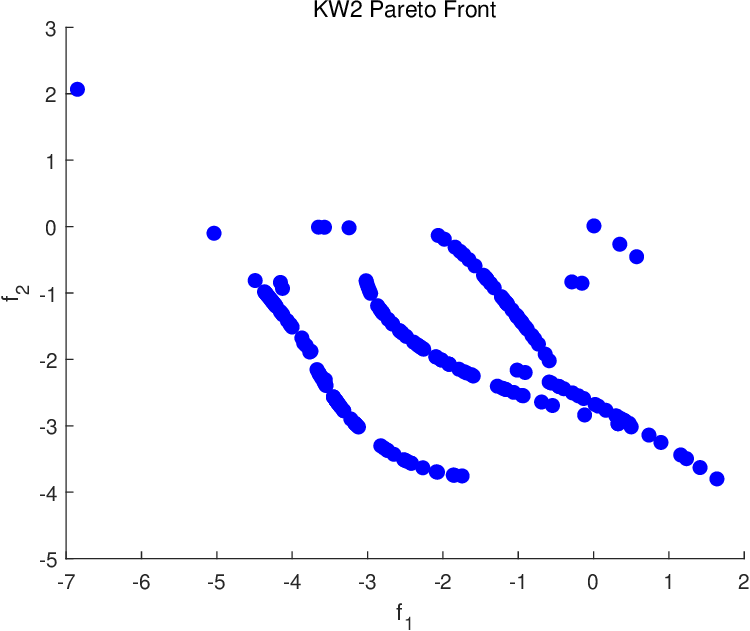}
		\caption{M-BFGS}
		\label{fig5:qnm}
	\end{subfigure}
	\hfill
	\begin{subfigure}[b]{0.3\textwidth}
		\centering
		\includegraphics[width=0.9\textwidth]{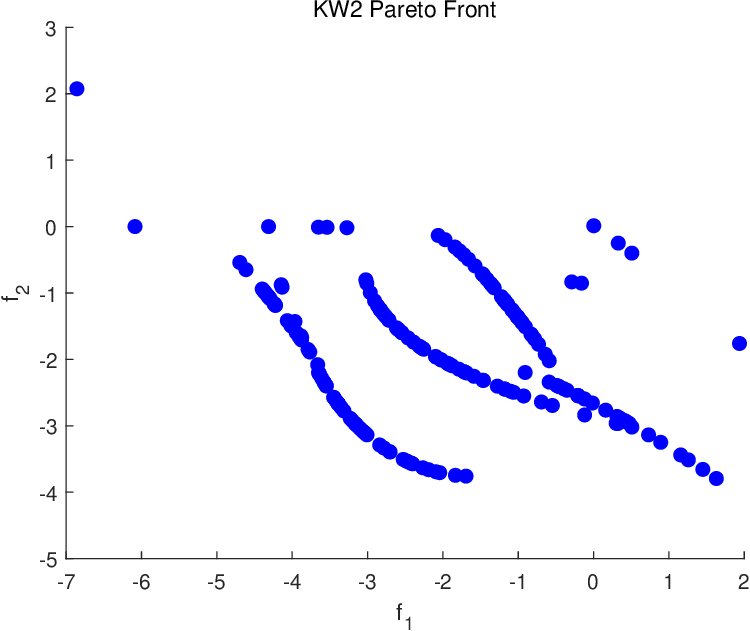}
		\caption{BB-DQN}
		\label{fig5:mbfgs}
	\end{subfigure}
	\hfill
	\begin{subfigure}[b]{0.3\textwidth}
		\centering
		\includegraphics[width=0.9\textwidth]{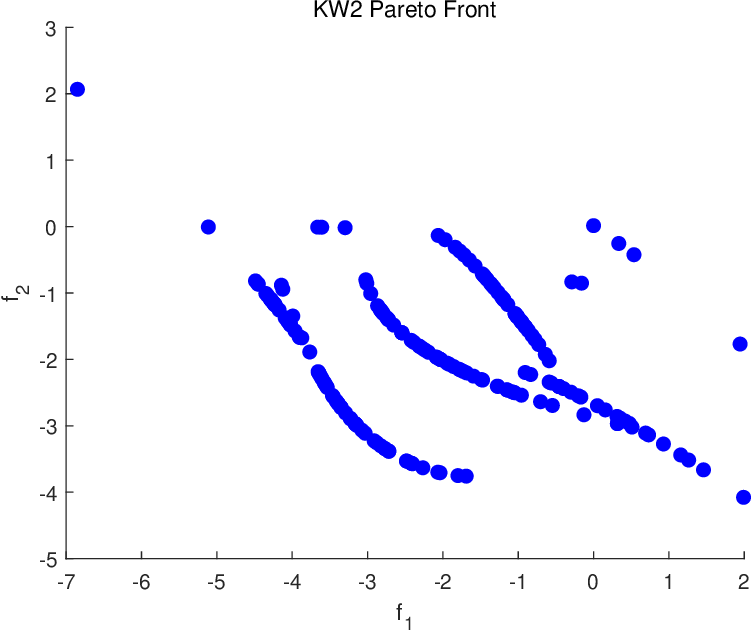}
		\caption{D-QN}
		\label{fig5:DQN}
	\end{subfigure}
	\caption{The approximated nondominated frontiers generated by M-BFGS, BB-DQN and D-QN for the KW2 problem}
	\label{fig5:pareto1}
\end{figure}

\begin{figure}[h!]
	\centering
	
	\begin{subfigure}[b]{0.3\textwidth}
		\centering
		\includegraphics[width=0.9\textwidth]{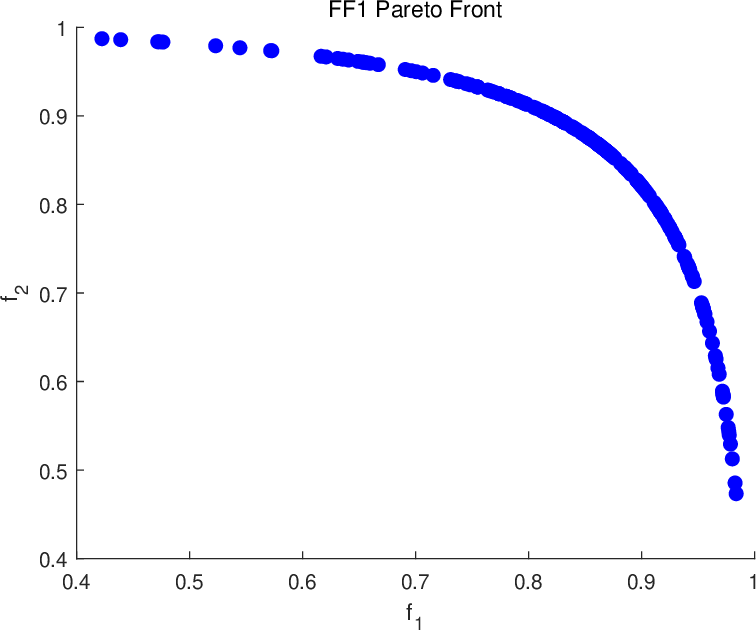}
		\caption{M-BFGS}
		\label{fig6:qnm}
	\end{subfigure}
	\hfill
	\begin{subfigure}[b]{0.3\textwidth}
		\centering
		\includegraphics[width=0.9\textwidth]{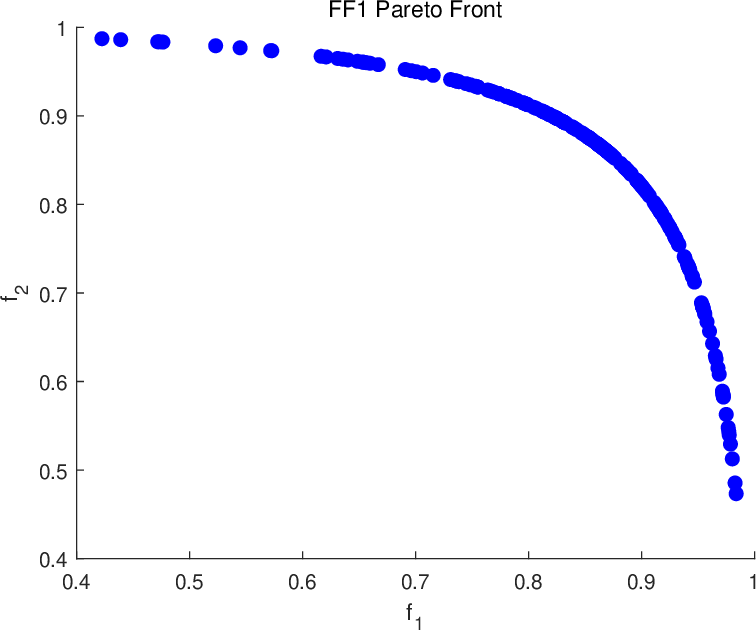}
		\caption{BB-DQN}
		\label{fig6:mbfgs}
	\end{subfigure}
	\hfill
	\begin{subfigure}[b]{0.3\textwidth}
		\centering
		\includegraphics[width=0.9\textwidth]{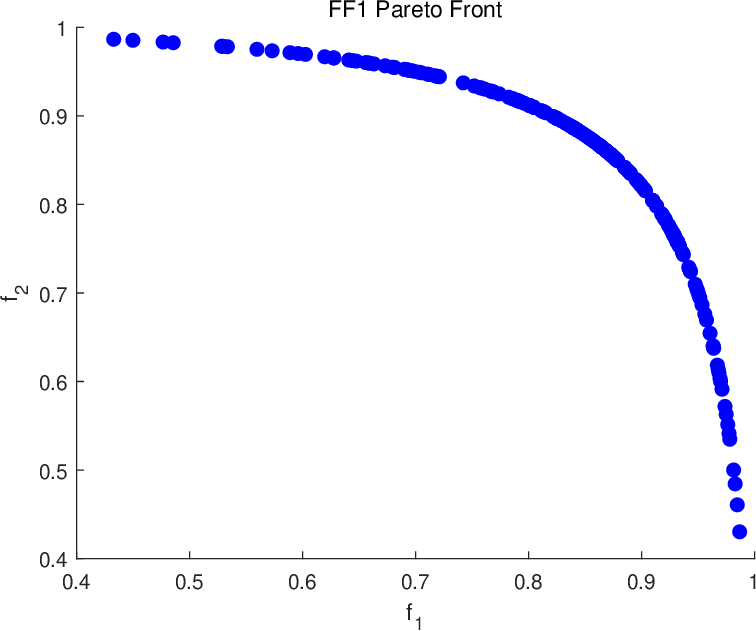}
		\caption{D-QN}
		\label{fig6:DQN}
	\end{subfigure}
	\caption{The approximated nondominated frontiers generated by M-BFGS, BB-DQN and D-QN for the FF1 problem}
	\label{fig6:pareto1}
\end{figure}

\begin{figure}[h!]
	\centering
	
	\begin{subfigure}[b]{0.3\textwidth}
		\centering
		\includegraphics[width=0.9\textwidth]{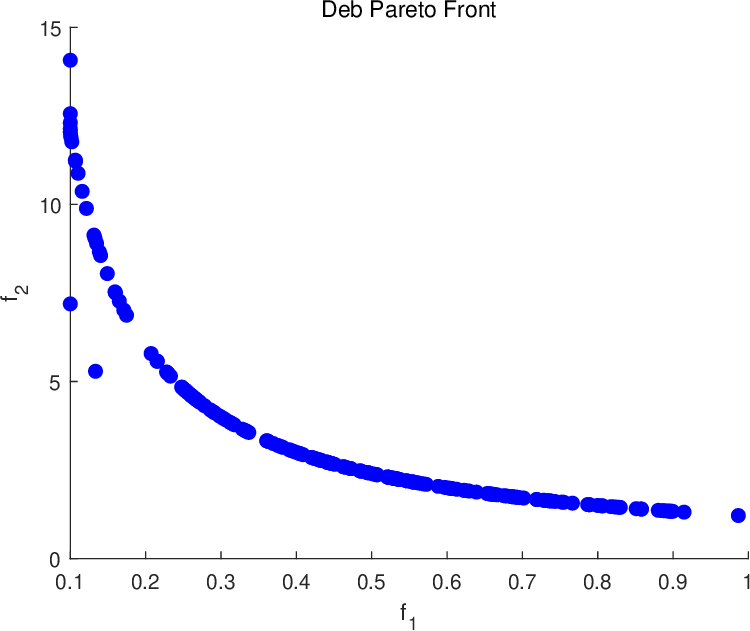}
		\caption{M-BFGS}
		\label{fig8:qnm}
	\end{subfigure}
	\hfill
	\begin{subfigure}[b]{0.3\textwidth}
		\centering
		\includegraphics[width=0.9\textwidth]{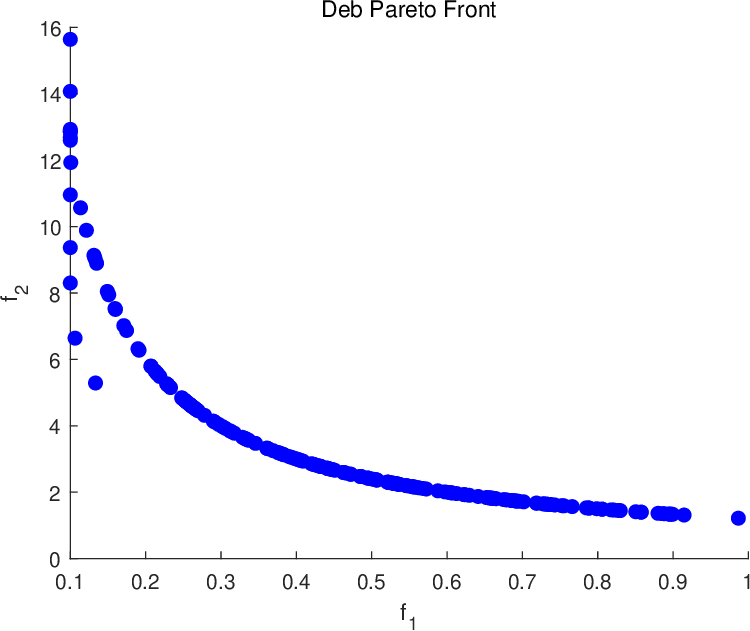}
		\caption{BB-DQN}
		\label{fig8:mbfgs}
	\end{subfigure}
	\hfill
	\begin{subfigure}[b]{0.3\textwidth}
		\centering
		\includegraphics[width=0.9\textwidth]{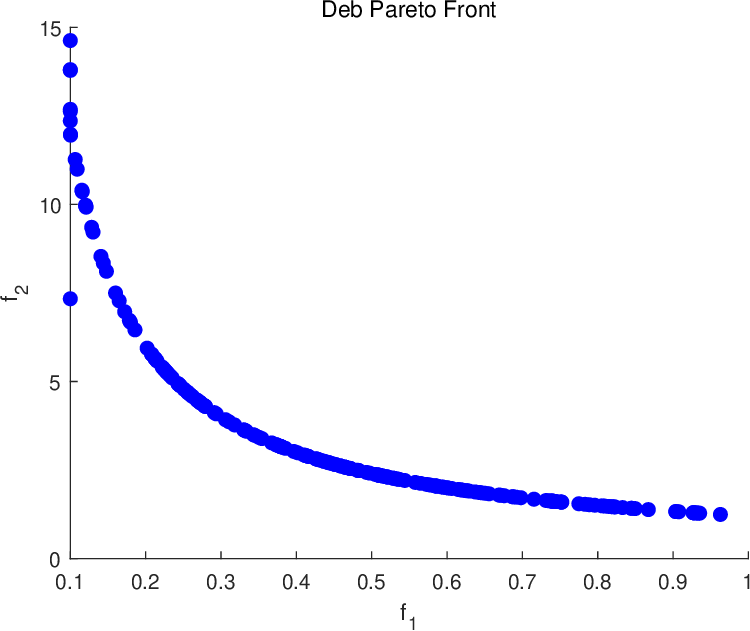}
		\caption{D-QN}
		\label{fig8:DQN}
	\end{subfigure}
	\caption{The approximated nondominated frontiers generated by M-BFGS, BB-DQN and D-QN for the Deb problem}
	\label{fig8:pareto1}
\end{figure}

\begin{figure}[h!]
	\centering
	\begin{subfigure}[b]{0.3\textwidth}
		\centering
		\includegraphics[width=0.9\textwidth]{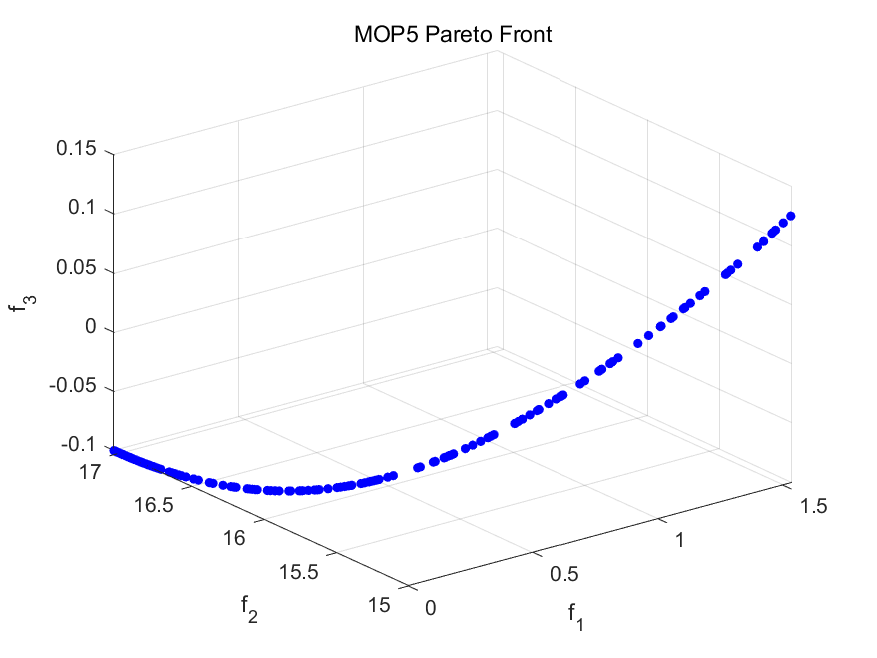}
		\caption{M-BFGS}
		\label{fig4:qnm}
	\end{subfigure}
	\hfill
	\begin{subfigure}[b]{0.3\textwidth}
		\centering
		\includegraphics[width=0.9\textwidth]{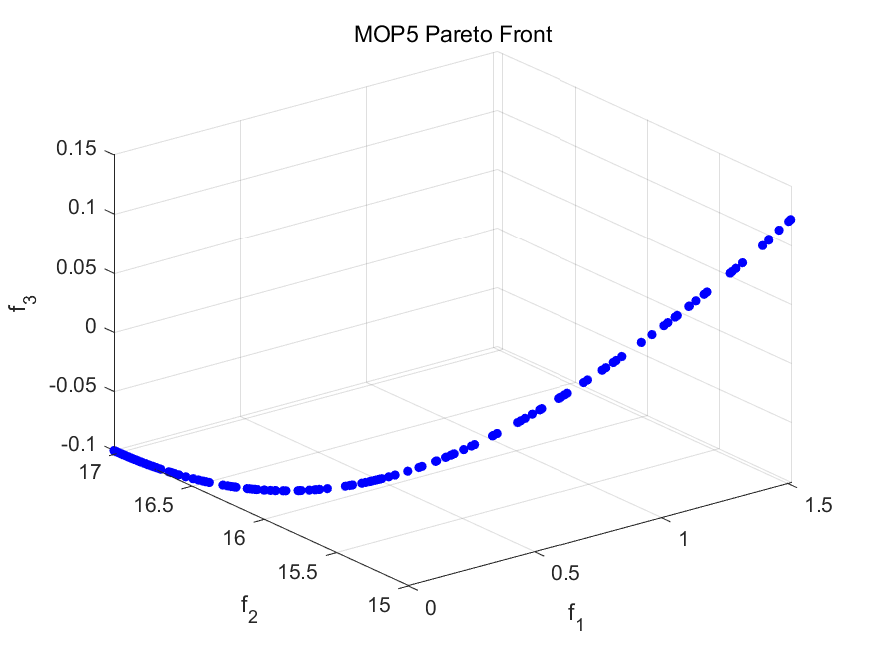}
		\caption{BB-DQN}
		\label{fig4:mbfgs}
	\end{subfigure}
	\hfill
	\begin{subfigure}[b]{0.3\textwidth}
		\centering
		\includegraphics[width=0.9\textwidth]{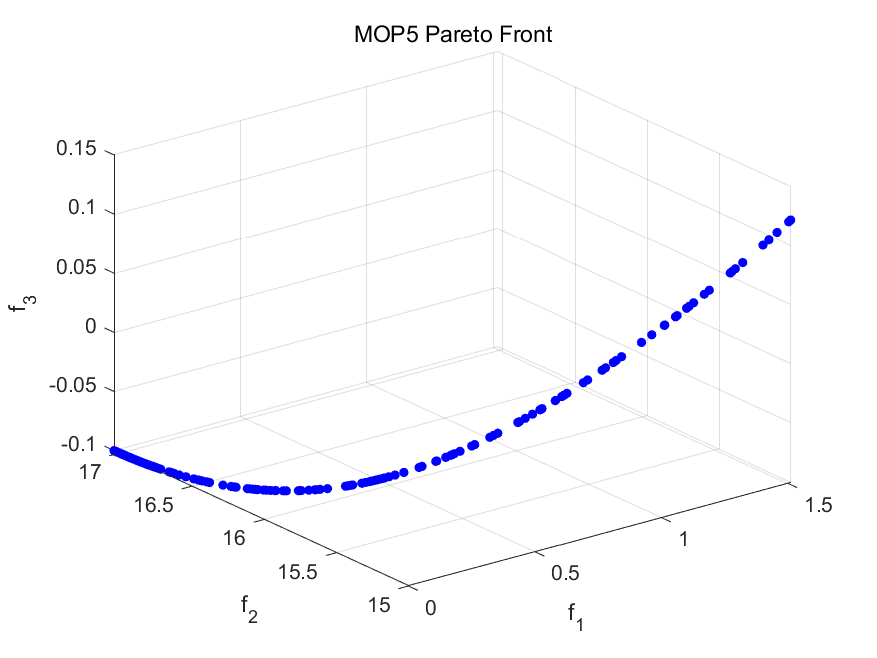}
		\caption{D-QN}
		\label{fig4:DQN}
	\end{subfigure}
	\caption{The approximated nondominated frontiers generated by M-BFGS, BB-DQN and D-QN for the MOP5 problem}
	\label{fig4:pareto1}
\end{figure}

\begin{figure}[h!]
	\centering
	
	\begin{subfigure}[b]{0.3\textwidth}
		\centering
		\includegraphics[width=0.9\textwidth]{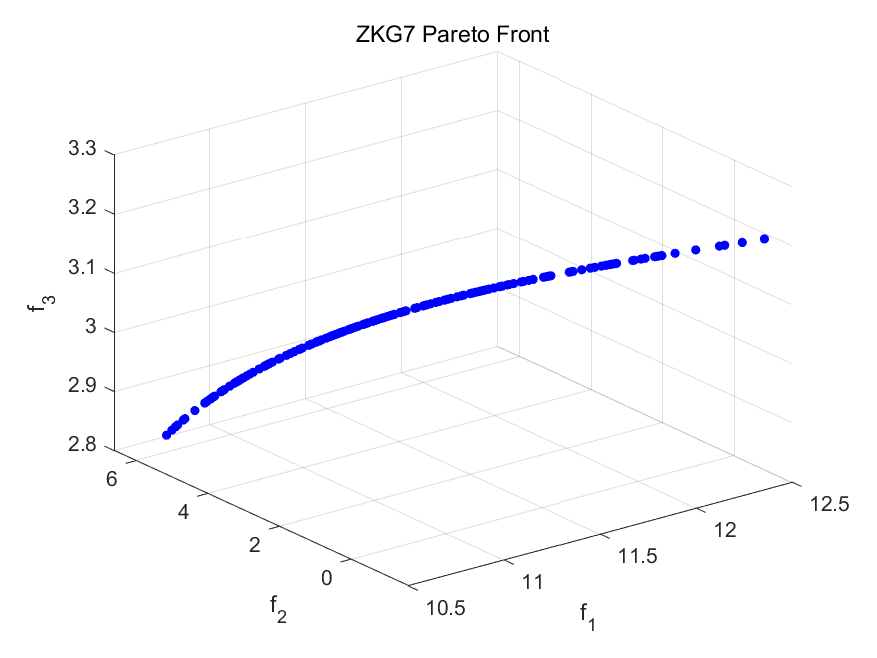}
		\caption{M-BFGS}
		\label{fig11:qnm}
	\end{subfigure}
	\hfill
	\begin{subfigure}[b]{0.3\textwidth}
		\centering
		\includegraphics[width=0.9\textwidth]{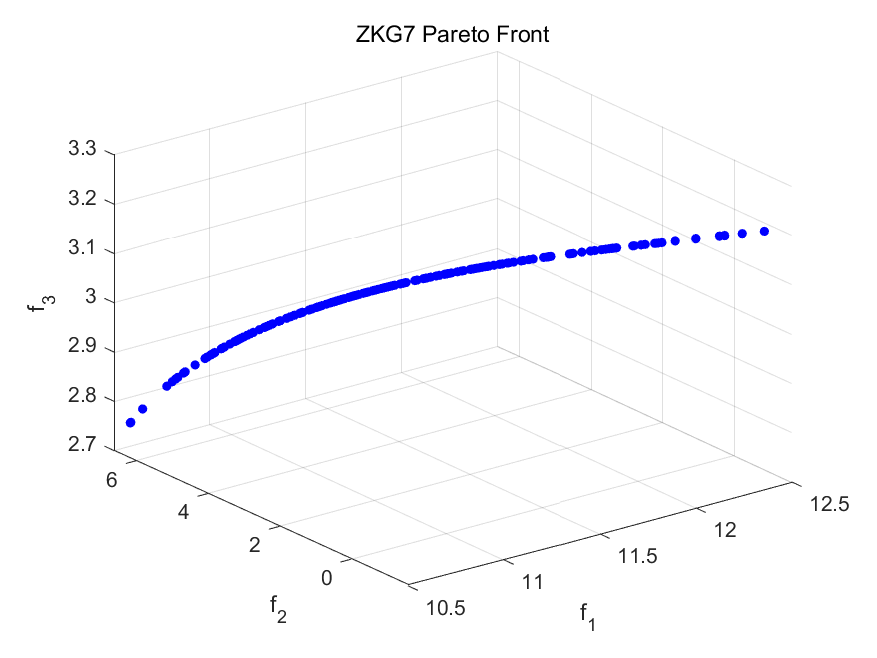}
		\caption{BB-DQN}
		\label{fig11:mbfgs}
	\end{subfigure}
	\hfill
	\begin{subfigure}[b]{0.3\textwidth}
		\centering
		\includegraphics[width=0.9\textwidth]{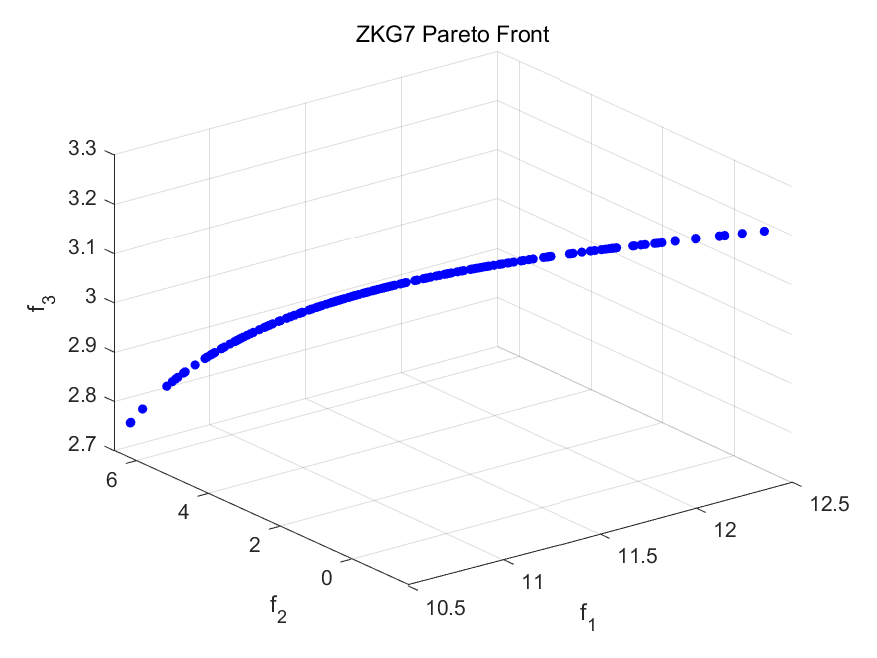}
		\caption{D-QN}
		\label{fig11:DQN}
	\end{subfigure}
	\caption{The approximated nondominated frontiers generated by M-BFGS, BB-DQN and D-QN for the ZKG7 problem}
	\label{fig11:pareto1}
\end{figure}

\begin{figure}[h!]
	\centering
	
	\begin{subfigure}[b]{0.3\textwidth}
		\centering
		\includegraphics[width=0.9\textwidth]{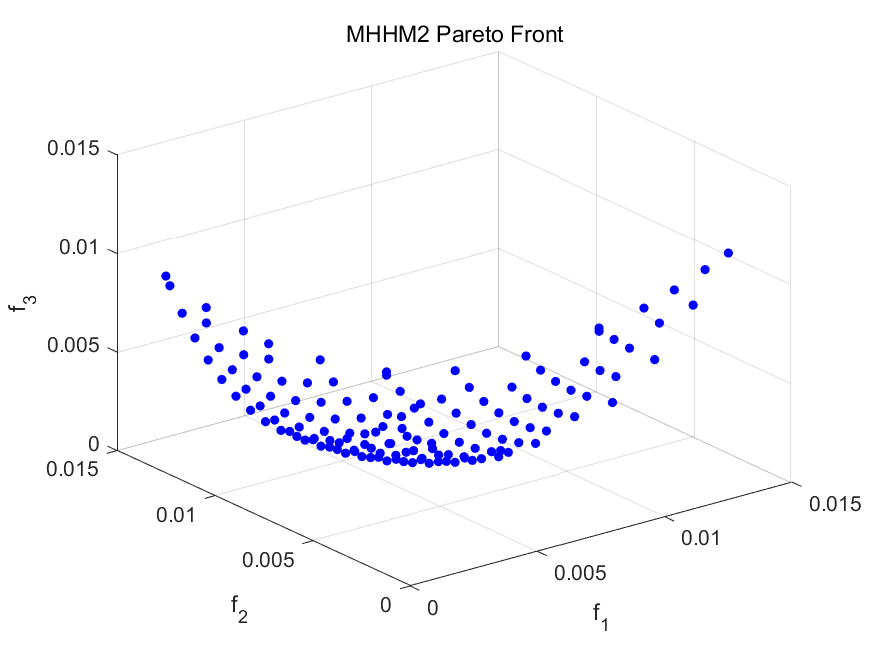}
		\caption{M-BFGS}
		\label{fig12:qnm}
	\end{subfigure}
	\hfill
	\begin{subfigure}[b]{0.3\textwidth}
		\centering
		\includegraphics[width=0.9\textwidth]{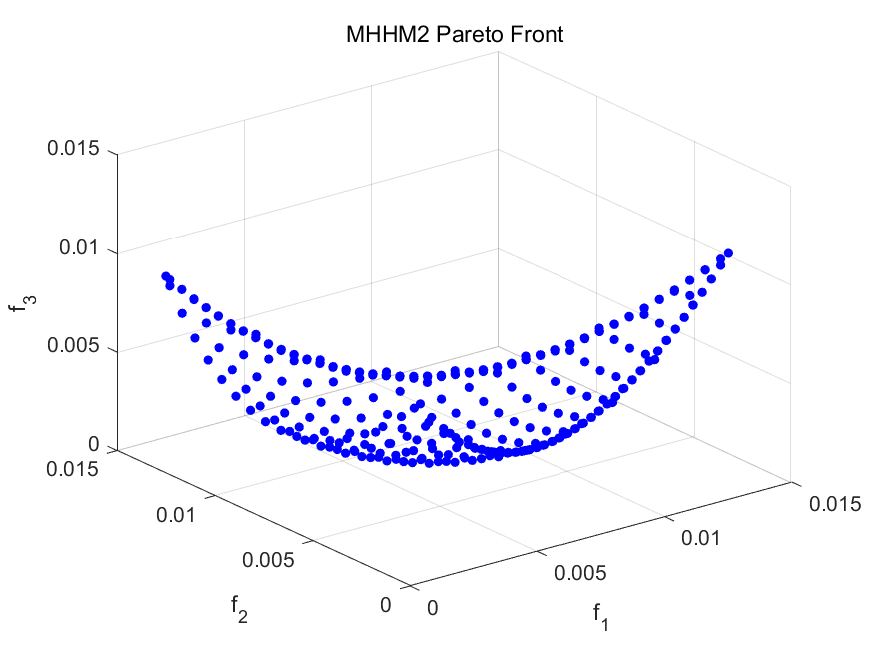}
		\caption{BB-DQN}
		\label{fig12:mbfgs}
	\end{subfigure}
	\hfill
	\begin{subfigure}[b]{0.3\textwidth}
		\centering
		\includegraphics[width=0.9\textwidth]{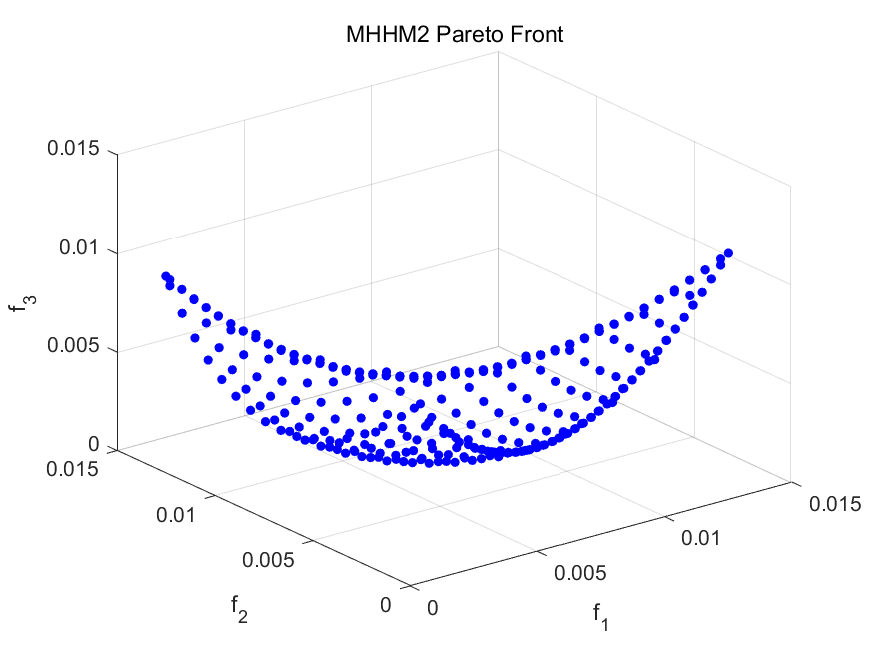}
		\caption{D-QN}
		\label{fig12:DQN}
	\end{subfigure}
	\caption{The approximated nondominated frontiers generated by M-BFGS, BB-DQN and D-QN for the MHHM2 problem with 2500 starting points}
	\label{fig12:pareto1}
\end{figure}

\section{Conclusion and Future Work}

This paper proposes the Barzilai-Borwein Diagonal Quasi-Newton (BB-DQN) method, an efficient quasi-Newton algorithm for solving nonconvex multiobjective optimization problems. The key innovation of the proposed method lies in adopting a unified modified Barzilai-Borwein diagonal matrix to approximate the aggregated Hessian matrix of all objective functions, which significantly reduces both computational and storage costs compared with methods that maintain separate Hessian approximations for each objective. Theoretical analysis verifies the reliability of BB-DQN: the algorithm globally converges to Pareto critical points without convexity assumptions, and its R-linear convergence rate is proved under standard conditions. Extensive numerical experiments demonstrate that BB-DQN achieves higher computational efficiency than the M-BFGS and D-QN methods on various benchmark problems, particularly in large-scale scenarios.

It is worth noting that imbalance is an inherent characteristic of multiobjective optimization problems. Therefore, in future work, we plan to integrate gradient normalization techniques and adaptive gradient scaling strategies based on the curvature differences of individual objective functions into the BB-DQN method. This will effectively address the issue of imbalance and further improve the efficiency of quasi-Newton algorithms for nonconvex multiobjective optimization problems.

\section*{Funding}
This research was supported by
the National Natural Science Foundation of China (Nos. 12171060, 12431010), the Chongqing Natural Science Foundation Project (No. CSTB2024NSCQ-LZX0140),
the Science and Technology Research Program of Chongqing Municipal Education Commission (No. KJZD-M202300504),
the Bayu Scholars Program of Chongqing Municipal Education Commission, and
Bowang Scholars Young Top Talents Cultivation Program of Chongqing Normal University.

\section*{Acknowledgement} These authors contributed equally to this work.

\section*{Disclosure statement}
The authors declare that they have no conflict of interest.


\end{document}